\documentclass{article}

\usepackage{graphicx}
\usepackage{amsmath}
\usepackage{amsthm}
\usepackage{amstext}
\usepackage{amsfonts}
\usepackage{amssymb}
\usepackage{amsxtra}
\usepackage{bbm}
\usepackage{url}
\usepackage[square,numbers,sort&compress]{natbib}
\usepackage{tikz}
\usepackage{algorithm}
\usepackage{algorithmic}

\newtheorem{thm}{Theorem}[section]

\newtheorem{lem}[thm]{Lemma}
\newtheorem{cor}[thm]{Corollary}
\theoremstyle{definition}

\newtheorem{definition}[thm]{Definition}

\newtheorem*{cexmp*}{Counterexample}

\newcommand \be {\begin{equation}}
\newcommand \ee {\end{equation}}
\newcommand \ben {\begin{equation*}}
\newcommand \een {\end{equation*}}
\newcommand \bea {\begin{eqnarray}}
\newcommand \eea {\end{eqnarray}}
\newcommand \bean {\begin{eqnarray*}}
\newcommand \eean {\end{eqnarray*}}

\newcommand \NN {\mathbb{N}}
\newcommand \RR {\mathbb{R}}

\newcommand \Rcal {\mathcal{R}}

\DeclareMathOperator*{\argmin}{argmin}
\DeclareMathOperator{\sign}{sign}

\DeclareMathOperator{\dom}{dom}
\DeclareMathOperator{\prox}{prox}

\newcommand{\scp}[2]{\langle #1 \vert #2\rangle}
\newcommand{\norm}[2][]{\|#2\|_{#1}}
\newcommand{\set}[2]{\{#1\, |\, #2\}}

\begin{document}

\title{The Linearized Bregman Method via Split Feasibility Problems:
  Analysis and Generalizations}

\author{Dirk~A. Lorenz\thanks{Institute for Analysis and Algebra, TU Braunschweig, 38092 Braunschweig, Germany, \texttt{d.lorenz@tu-braunschweig.de}} \and Frank~Sch\"{o}pfer\thanks{Institut f\"ur Mathematik, Carl von Ossietzky Universit\"at Oldenburg, 26111 Oldenburg, Germany, \texttt{frank.schoepfer@uni-oldenburg.de}} \and Stephan Wenger\thanks{Institute for Computer Graphics, TU Braunschweig, 38092 Braunschweig, Germany, \texttt{wenger@cg.cs.tu-bs.de}}}

\maketitle

\begin{abstract}
  The linearized Bregman method is a method to calculate sparse solutions to systems of linear equations.
  We formulate this problem as a split feasibility problem, propose an algorithmic framework based on Bregman projections and prove a general convergence result for this framework.
  Convergence of the linearized Bregman method will be obtained as a special case.
  Our approach also allows for several generalizations such as other objective functions, incremental iterations, incorporation of non-gaussian noise models or box constraints.
\end{abstract}



\section{Introduction}
\label{sec:intro}
We consider $A\in\RR^{m\times n}$, with $m<n$ and full rank, $b\in\RR^m$ and aim at solutions of the underdetermined linear system $Ax=b$.
The least squares solution is obtained as the solution with minimal 2-norm while it is known that sparse solutions are obtained as solutions with minimal 1-norm, an approach that has been coined \emph{Basis Pursuit} in~\cite{CDS98}.
The linearized Bregman method, introduced in~\cite{YOGD08}, solves a regularized version of the Basis Pursuit problem with regularization parameter $\lambda >0$
\be
\label{eq:sparse-lin-sys}
\min_{x\in\RR^n} \lambda\norm[1]{x} + \tfrac{1}{2}\norm[2]{x}^2 \quad\text{s.t.} \quad Ax=b.
\ee
It consists of the simple iteration
\bean
\label{eq:linbreg}
x^k &=& S_\lambda(v^{k-1})\\
v^k &=& v^{k-1} - A^T(Ax^k - b)
\eean
initialized with $x^0=v^0=0$, where $S_\lambda (x) = \min(|x|-\lambda,0)\sign(x)$ is the component-wise soft shrinkage. 
Convergence of the method has been analyzed in~\cite{COS09} and~\cite{Yin10}.
The method has been derived from the Bregman iteration~\cite{OBGXY05} and also identified as a gradient descent for the dual problem of~\eqref{eq:sparse-lin-sys} in~\cite{Yin10}.

In this paper we provide a new convergence proof by phrasing~\eqref{eq:sparse-lin-sys} as a \emph{split feasibility problem}~\cite{CE94}.
In this framework the linearized Bregman method will be a special case of a broader class of methods including other popular methods like the Landweber method~\cite{Lan51} or the Kaczmarz method~\cite{K37}.
Given a finite number of convex sets $C_i\in\RR^n$, $i=1,\dots,k_C$, $Q_j\in\RR^{m_j}$ and matrices $A_j\in\RR^{m_j \times n}$, $j=1,\dots,n_Q$, a split feasibility problem is to find a vector $x\in\RR^n$ such that
\be
\label{eq:sfp}
x\in C := \set{x\in\RR^n}{x\in C_i,\ i=1,\dots,n_C\ \text{and}\ A_j x\in Q_j,\ j=1,\dots,n_Q}.
\ee
Numerous iterative methods have been proposed to solve~\eqref{eq:sfp} which are based on successive orthogonal projections onto the individual sets $C_i$ and $Q_j$~\cite{Com96,BB96,Byr02,ZY05}.
By employing the more general \emph{Bregman projections} with respect to some given function $f:\RR^n \to \RR$ one can steer the iterates in such a way that their \emph{Bregman distance} with respect to $f$ is decreasing~\cite{BB97,BBC03,Bre67}.
In some instances the iterates then even converge to solutions of the more ambitious problem
\be
\label{eq:f-sfp}
\min_{x\in\RR^n} f(x)\quad\text{s.t.}\quad x\in C
\ee
with the above set $C$.
Bregman projections are also used to solve feasibility problems in infinite-dimensional Banach spaces~\cite{AB97,SSL08b}.

Formally we can phrase the constraint $Ax=b$ of problem~\eqref{eq:sparse-lin-sys} in the form~\eqref{eq:sfp} in at least two different ways:
\begin{itemize}
\item We set $n_C=0$ (i.e. there is no set $C_i$) and $n_Q=1$
  and use $A_1=A$ and $Q = \{b\}$.
\item We set $n_Q=0$ (i.e. there is no set $Q_j$), $n_C=m$ and the sets $C_i$ are the hyperplanes
  \ben
  C_i = \set{x\in\RR^n}{\scp{a_i}{x} = b_i}
  \een
  given by the rows $a_i$ of $A$ and the $i$th components $b_i$ of $b$.
\end{itemize}
We will see that an iterative method employing Bregman projections with respect to $f(x)=\lambda\norm[1]{x} + \tfrac{1}{2}\norm[2]{x}^2$ in the first formulation corresponds to the linearized Bregman method, whereas the second formulation will lead to a kind of ``sparse'' Kaczmarz method.
Moreover it is easy to incorporate additional prior knowledge like positivity as a convex constraint $C_i := \set{x\in\RR^n}{x \ge 0}$, or to deal with noisy data $b^\delta$ by setting $Q_j:=\set{y\in\RR^n}{\|y-b^\delta\| \le \delta}$, where the norm is adapted to the noise.
However, previous convergence results of iterative methods employing Bregman projections with respect to $f$ require $f$ to be smooth.
Hence the case $f(x) = \lambda\norm[1]{x} + \tfrac{1}{2}\norm[2]{x}^2$ is not yet covered.
In section~\ref{sec:BPSFP} we extend the convergence analysis to functions $f$ that are only required to be continuous and strongly convex.
Convergence of the linearized Bregman method and its extensions will then be obtained as a special case.
In section~\ref{sec:num-experiments} we illustrate the performance of the method for sparse recovery problems with noisy data, for the problem of three dimensional reconstruction of planetary nebulae and for a tomographic reconstruction problem with several constraints.

\section{Bregman projection algorithms for split feasibility problems}
\label{sec:BPSFP}
In this section we state the framework for the solution of~\eqref{eq:sfp}, ~\eqref{eq:f-sfp} and prove convergence of the derived algorithms.

\subsection{Basic assumptions and notions}
\label{sec:assumptions}

Let $f:\RR^n \to \RR$ be continuous and convex with conjugate function $f^*:\RR^n \to \RR$,
\be
f^*(x^*)= \sup_{x \in \RR^n} \scp{x^*}{x} - f(x). \label{eq:conjugate}
\ee
Since $\dom(f)=\RR^n$ the conjugate function $f^*$ is coercive, i.e.
\be
\lim_{\|x^*\|_2 \to \infty} \frac{f^*(x^*)}{\|x^*\|_2}= \infty. \label{eq:coercive}
\ee
By $\partial f(x)$ we denote the subdifferential of $f$ at $x \in \RR^n$,
\be
\partial f(x) = \set{x^* \in \RR^n}{ f(y) \ge f(x) + \scp{x^*}{y-x}\quad \mbox{for all} \quad y \in \RR^n }. \label{eq:subdiff}
\ee
We assume that $f$ is strongly convex, i.e. there exists some constant $\alpha>0$ such that
\be
f(y) \ge f(x) + \scp{x^*}{y-x} + \frac{\alpha}{2} \norm[2]{y-x}^2 \quad \mbox{for all} \quad x^* \in \partial f(x). \label{eq:strongly convex}
\ee
Especially $f$ is strictly convex.
Note that strong convexity also implies coercivity of $f$.
Furthermore, by~\cite[Prop.~12.60]{RW09}, the conjugate function $f^*$ is differentiable with a Lipschitz-continuous gradient,
\be
\norm[2]{\nabla f^*(x^*)-\nabla f^*(y^*)} \le \frac{1}{\alpha}\norm[2]{x^*-y^*} \quad \mbox{for all} \quad x^*,y^* \in \RR^n. \label{eq:Lipschitz} 
\ee
and the following inequality holds
\be
f^*(y^*) \le f^*(x^*) + \scp{\nabla f^*(x^*)}{y^*-x^*} + \frac{1}{2 \alpha} \norm[2]{y^*-x^*}^2 \quad \mbox{for all} \quad x^*,y^* \in \RR^n. \label{eq:remainder_inequality}
\ee
The \textit{Bregman distance} (cf.~\cite{BR06}) $D^{x^*}(x,y)$ between $x,y \in \RR^n$ with respect to $f$ and a subgradient $x^* \in \partial f(x)$ is defined as,
\be
D^{x^*}(x,y):=f(y)-f(x) -\scp{x^*}{y - x}\label{eq:bregman-dist}
\ee
Note that for $f(x)=\frac{1}{2}\norm[2]{x}^2$ we just have $D^{x^*}(x,y)=\frac{1}{2}\norm[2]{x-y}^2$.
In general $D^{x^*}$ is not a distance function in the usual sense, as it is in general neither necessarily symmetric, nor positive definite and does not have to obey a (quasi-)triangle inequality.
Nevertheless it has some distance-like properties which we state in the following lemma.
They are immediately clear from our basic assumptions.

\begin{lem} \label{lem:D}
Let $f$ fulfill the basic assumptions stated above.
For all $x,y \in \RR^n$ and $x^* \in \partial f(x)$ we have
\[
D^{x^*}(x,y) \: \ge \:  \frac{\alpha}{2} \norm[2]{x-y}^2 \:\ge \: 0
\]
and
\[
D^{x^*}(x,y) = 0 \quad \Leftrightarrow \quad x=y.
\]
For $x,x^*$ fixed the function $y \mapsto D^{x^*}(x,y)$ is continuous, coercive and strongly convex with
\[\partial_y D^{x^*}(x,y) = \partial f(y)-x^*.
\]
\end{lem}

Closely related to the Bregman distance is the following function $\Delta$ which involves $f^*$,
\be
\Delta(x^*,x) :=f^*(x^*)-\scp{x^*}{x} + f(x) \quad \mbox{for arbitrary} \quad x^*,x \in \RR^n. \label{eq:Delta}
\ee
Due to the properties of $f$ and $f^*$ the function $\Delta$ is continuous, convex and coercive in both arguments.
By~\cite[Prop.~11.3]{RW09} we have $\Delta(x^*,x) \ge 0$ and
\be
\Delta(x^*,x) = 0 \quad \Leftrightarrow \quad x^* \in \partial f(x) \quad \Leftrightarrow \quad x = \nabla f^*(x^*). \label{eq:Delta-subgradients}
\ee
Therefore, the Bregman distance is related to $\Delta$ by
\be
D^{x^*}(x,y)=\Delta(x^*,y) \quad \mbox{for all} \quad x^* \in \partial f(x). \label{eq:D-and-Delta}
\ee

Note that the assumption of strong convexity is not too severe.
A common way to enforce strong convexity is to regularize the function $f$ by adding a strongly convex functional.
Probably the simplest approach is to use $f(x) + \tfrac{1}{2\lambda}\norm[2]{x}^2$ or $\lambda f(x) + \tfrac12\norm[2]{x}^2$, like in~\eqref{eq:sparse-lin-sys}.
In this case the gradient of the conjugate function involves the proximal mapping of $f$,
\ben
  \nabla(\lambda f + \tfrac12\norm[2]{\cdot}^2)^*(x) =   \prox_{\lambda f}(x) = \argmin_{y \in \RR^n} \lambda f(y) + \tfrac12\norm[2]{x-y}^2 \,.
\een

\subsection{Bregman projections}

Let $C \subset \RR^n$ be a nonempty closed convex set, $x \in \RR^n$ and $x^* \in \partial f(x)$. 
The \textit{Bregman projection} of $x$ onto $C$ with respect to $f$ and $x^*$ is the point $\Pi_C^{x^*}(x) \in C$ such that
\be
D^{x^*}\big(x,\Pi_C^{x^*}(x)\big) = \min_{y \in C} D^{x^*}(x,y) \,. \label{eq:Pi}
\ee
Lemma~\ref{lem:D} guarantees that the Bregman projection exists and is unique.
Note that for $f(x)=\frac{1}{2}\|x\|_2^2$ this is just the orthogonal projection onto $C$.
To distinguish this case we denote the orthogonal projection by $P_C(x)$.
Note that the notation for the Bregman projection does not capture its dependence on the function $f$, which, however, will always be clear from the context.
The next lemma characterizes the Bregman projection by a variational inequality.

\begin{lem} \label{lem:Pi}
A point $z \in C$ is the Bregman projection of $x$ onto $C$ with respect to $f$ and $x^* \in \partial f(x)$ iff there is some $z^* \in \partial f(z)$ such that one of the following equivalent conditions is fulfilled
\be
\scp{z^*-x^*}{y-z} \ge 0 \quad \mbox{for all} \quad y \in C  \label{eq:Pi-variational}
\ee
\be
D^{z^*}(z,y) \le  D^{x^*}(x,y) - D^{x^*}(x,z) \quad \mbox{for all} \quad y \in C \,. \label{eq:Pi-D}
\ee
We call any such $z^*$ an \emph{admissible subgradient} for $z=\Pi_C^{x^*}(x)$.
\end{lem}

\begin{proof}
By theorem 3.33 in~\cite{Rus06} a point $z \in C$ minimizes $D^{x^*}(x,y)$ among all $y \in C$ iff there is some $u^* \in \partial_y D^{x^*}(x,z)$ such that
\[
\scp{ u^*}{ y-z}\ge 0 \quad \mbox{for all} \quad y \in C \,.
\]
By lemma~\ref{lem:D} we have $\partial_y D^{x^*}(x,z)= \partial f(z)-x^*$, which yields~\eqref{eq:Pi-variational}.
Equivalence to~\eqref{eq:Pi-D} is straightforward by using the definition of $D$.
\end{proof}

As far as we know lemma~\ref{lem:Pi} and the following lemma~\ref{lem:Pi_affine_space_hyperplane} have not yet been considered for non-differentiable functions $f$---for differentiable  $f$ we also refer to~\cite{BR06, SSL08a}.
Since Bregman projections onto affine subspaces and half-spaces will be the backbone of our algorithms, it is important to know how to compute them efficiently.

\begin{definition}
  \label{def:aff-space-half-space}
  Let $A \in \RR^{m \times n}$ have full row rank, $b \in \RR^m$,
  $0\not=a \in \RR^n$ and $\beta \in \RR$.
  
  By $L(A,b)$ we denote the \emph{affine subspace}
  \[
  L(A,b):=\set{ x \in \RR^n}{ Ax = b } \,,
  \]
  by $H(a,\beta)$ the \emph{hyperplane}
  \[
  H(a,\beta):=\set{ x \in \RR^n}{ \scp{ a}{ x }= \beta }
  \]
  and by $H_{\le}(a,\beta)$ the \emph{half-space}
  \[
  H_{\le}(a,\beta) :=\set{ x \in \RR^n }{ \scp{a}{x} \le \beta } \,.
  \]
\end{definition}

\begin{lem}\label{lem:Pi_affine_space_hyperplane}
  Let $f$ fulfill the basic assumptions and $A$, $b$, $a$ and $\beta$ be as in definition~\ref{def:aff-space-half-space}.
\begin{itemize}
\item[(a)] The Bregman projection of $x \in \RR^n$ onto $L(A,b)$ is
\[
z:=\Pi_{L(A,b)}^{x^*}(x)=\nabla f^*(x^*-A^T \hat{w}) \,,
\]
where $\hat{w}\in \RR^m$ is a solution of
\[
\min_{w \in \RR^m} f^*(x^*-A^T w) + \scp{ w }{ b } \,.
\]
Moreover, an admissible subgradient for $z$ is $z^*:=x^*-A^T \hat{w}$ and for all $y\in L(A,b)$ we have
\be
D^{z^*}(z,y) \le D^{x^*}(x,y) - \frac{\alpha}{2} \norm[2]{(AA^T)^{-\frac{1}{2}}(Ax-b)}^2 \,. \label{eq:D-decreasing-L}
\ee
\item[(b)] The Bregman projection of $x \in \RR^n$ onto $H(a,\beta)$ is
\[
z := \Pi_{H(a,\beta)}^{x^*}(x)=\nabla f^*(x^*- \hat{t} \cdot a) \,,
\]
where $\hat{t} \in \RR$ is a solution of
\be
\label{eq:exact-linesearch-BP-halfspace}
\min_{t \in \RR} f^*(x^*-t  \cdot a) + t \cdot \beta \,.
\ee
Moreover, an admissible subgradient for $\Pi_{H(a,\beta)}^{x^*}(x)$ is $x^* - \hat{t} \cdot a$ and for all $y\in H(a,\beta)$ we have
\be
D^{z^*}(z,y) \le D^{x^*}(x,y) - \frac{\alpha}{2} \frac{(\scp{ a }{ x } - \beta)^2} {\norm[2]{a}^2} \,. \label{eq:D-decreasing-H}
\ee
If $x$ is not contained in $H_{\le}(a,\beta)$ then we necessarily have $\hat{t}>0$ and in this case also $\Pi_{H_{\le}(a,\beta)}^{x^*}(x)=\Pi_{H(a,\beta)}^{x^*}(x)$.
\end{itemize}
\end{lem}
\begin{proof}
(a) Note that the function $g:\RR^m \to \RR$ with $g(w)=f^*(x^*-A^T w) + \scp{ w }{ b } $ is convex, differentiable and coercive.
Hence $g$ attains its minimum at some $\hat{w} \in \RR^m$ with
\ben
\nabla g(\hat{w})=0 \quad \Leftrightarrow \quad A\, \nabla f^*(x^*-A^T \hat{w}) = b \,,
\een
which yields $z:=\nabla f^*(x^*-A^T \hat{w}) \in L(A,b)$.
Since $z$ and $z^*:=x^*-A^T \hat{w}\in \partial f(z)$ fulfill the variational inequality~\eqref{eq:Pi-variational} we indeed have $z=\Pi_{L(A,b)}^{x^*}(x)$ and $z^*$ is admissible.
Furthermore by~\eqref{eq:Delta},~\eqref{eq:D-and-Delta} and~\eqref{eq:remainder_inequality} we can estimate for all $y\in L(A,b)$ and $w \in \RR^m$
\bean
D^{z^*}(z,y) &=& f^*(x^*-A^T \hat{w}) + \scp{ \hat{w}}{ b }  - \scp{x^*}{y} + f(y) \\
& \le & f^*(x^*-A^T w) + \scp{ w}{ b}  - \scp{ x^* }{ y} + f(y) \\
& \le & f^*(x^*) - \scp{x}{A^T w} + \frac{1}{2\alpha} \norm[2]{A^T w}^2 + \scp{ w }{b}  - \scp{x^*}{y} + f(y) \\
&=& D^{x^*}(x,y) - \scp{A x -b}{ w} + \frac{1}{2\alpha} \norm[2]{A^T w}^2 
\eean
The right hand side of the above inequality becomes minimal for
\[
\tilde{w}=\alpha \cdot (AA^T)^{-1}(Ax-b)
\]
and by inserting $\tilde{w}$ we arrive at inequality~\eqref{eq:D-decreasing-L}.\\
The assertion (b) for hyperplanes is a corollary to (a).
Now assume that $x$ is not contained in the halfspace $H_{\le}(a,\beta)$.
Since the function $g(t)=f^*(x^*-t  \cdot a) + t \cdot \beta$ is increasing with $g'(0)=\beta - \scp{a}{x} < 0$ we must have $\hat{t}>0$.
\end{proof}

For the special case $f(x)=\lambda \|x\|_1 + \frac{1}{2} \|x\|_2^2$ we can also explicitly compute the Bregman projection onto the positive cone or a box.
The proof is straightforward by checking the variational inequality~\eqref{eq:Pi-variational}.

\begin{lem}\label{lem:Pi_positive_box}
  Let $f(x)=\lambda \|x\|_1 + \frac{1}{2} \|x\|_2^2$.
\begin{itemize}
\item[(a)] Let $C_+=\RR^n_{\ge 0}$ be the positive cone.
Then we have
\[
\Pi_{C_+}^{x^*}(x)=S_\lambda\big(P_{C_+}(x^*)\big) = \left\{ \begin{array}{c@{\quad,\quad}l} x^*_j - \lambda & x^*_j > \lambda \\ 0 & x^*_j \le  \lambda  \end{array} \right.
\]
and $P_{C_+}(x^*)$ is an admissible subgradient for $\Pi_{C_+}^{x^*}(x)$.

\item[(b)] Let $B=\prod_{j=1}^n [a_j,b_j]$ be a box with $0 \in B$.
Then we have
\[
z:=\Pi_B^{x^*}(x)=P_B\big(S_\lambda(x^*)\big)
\]
and an admissible subgradient $z^* \in \partial f(z)$ is given by
\[
z^*_j:= \left\{ \begin{array}{c@{\quad,\quad}l} x^*_j &  a_j \le (S_\lambda(x^*))_j \le b_j \\
																							 b_j+\lambda  & (S_\lambda(x^*))_j > b_j \\
																							 a_j-\lambda  & (S_\lambda(x^*))_j < a_j \end{array} \right.
\]
We may simplify $z^*$ by setting $z^*_j:= 0$ in case $z_j=0$ and either $a_j=0, x^*_j<0$ or $b_j=0, x^*_j>0 $.
\end{itemize}
\end{lem}

We note that for the function $f(x)=\lambda \|x\|_1^2 + \frac{1}{2} \|x\|_2^2$ an analoguous result to (a) holds with the shrinkage operation $S_\lambda$ replaced by the relative shrinkage operation $S_{c(\lambda)}$, see~\cite{Sch12}.

\subsection{The method of Bregman projections for split feasibility problems}
\label{sec:breg-proj-sfp}

Consider a \emph{convex feasibility problem} (CFP),
\be
\mbox{find} \quad x \in C=\bigcap_{i=1,\ldots,m} C_i \tag{CFP}\label{eq:CFP}
\ee
with closed convex sets $C_i \subset \RR^n$ such that the intersection $C$ is not empty.
A simple and widely known idea to solve~\eqref{eq:CFP} is to project successively onto the individual sets $C_i$ and we refer to~\cite{BB96} for an excellent introduction.
By now there is a vast literature on CFPs and projection algorithms for their solution, see e.g.~\cite{BB97, BBC03, Bre67, Byr04, CEKB05, SSL08b, ZY05}.
These projection algorithms are most efficient if the projections onto the individual sets are relatively cheap.
A special instance of the CFP is the \emph{split feasibility problem} (SFP)~\cite{CE94,BC01,Byr02}, where some of the sets $C_i$ arise by imposing convex constraints $Q_i \subset \RR^{m_i}$ in the range of a matrix $A_i \in \RR^{m_i \times n}$,
\be
C_i=C_{Q_i} =\{ x \in \RR^n \,|\, A_i x \in Q_i \} \,. \label{eq:Q_i}
\ee
In general projections onto such sets can be prohibitively expensive and we call the constraints $A_ix\in Q_i$ \emph{difficult} (in contrast to \emph{simple} constraints $x\in C_i$).
Hence, it is often preferable to use projections onto suitable enclosing halfspaces.
The following lemma shows a construction of such an enclosing halfspace.

\begin{lem} \label{lem:enclosing-halfspace}
Assume that $\tilde{x} \notin C_Q=\set{ x \in \RR^n}{ A x \in Q }$ and set
\[
w:=A\tilde{x}-P_Q(A\tilde{x}) \quad \mbox{and} \quad \beta:=\scp{A^T w}{ \tilde{x}} - \norm[2]{w}^2 \,.
\]
Then it holds that $A^T w \not=0$, $\tilde{x} \notin H_{\le}(A^T w, \beta)$ and $C_Q \subset H_{\le}(A^T w, \beta)$, in other words, the hyperplane $H(A^T w,\beta)$ separates $\tilde{x}$ from $C_Q$.
\end{lem}

\begin{proof}
The assumption $\tilde{x} \notin C_Q$ is equivalent to $w \not= 0$.
Hence we have $\scp{A^T w}{\tilde{x}} > \beta$.
Moreover for all $x \in C_Q$ we have $\scp{w}{Ax-P_Q(A\tilde{x})} \le 0$ and thus can estimate 
\bean
\scp{ A^T w}{x} &=&  \scp{ w}{Ax-P_Q(A\tilde{x})} + \scp{w}{P_Q(A\tilde{x}) - A\tilde{x}} + \scp{A^T w}{\tilde{x}} \\
& \le & \scp{ A^T w}{\tilde{x}}- \norm[2]{w}^2 = \beta \,.
\eean
If $A^T w =0$ this would imply $\|w\|_2=0$ which contradicts $\tilde{x} \notin C_Q$.
\end{proof}

To solve a split feasibility problem one can proceed as follows:
Encounter the different constraints $C_i$ and $C_{Q_i}$ successively.
For a simple constraint $C_i$  project the current iterate onto $C_i$, while for a difficult constraint $C_{Q_i}$ project the current iterate onto a separating hyperplane according to lemma~\ref{lem:enclosing-halfspace}.
To formalize this idea we introduce some notation.
Set $I:=\{1,\ldots,m\}$, let $I_Q \subset I$ be the subset of all indices $i$ belonging to difficult constraints $C_{Q_i}$, and denote by $I_C:=I \setminus I_Q$ the set of the remaining indices.
The split feasibility is then formulated as
\be
\mbox{find} \quad x\quad\mbox{such that}\quad x\in C_i\ \mbox{for}\ i\in I_C\quad\mbox{and}\quad A_i x\in Q_i\ \mbox{for}\ i\in I_Q \tag{SFP}\label{eq:SFP}.
\ee
Further let $r:\NN\to I$ be a \emph{control sequence}, i.e. $r(k)$ indicates which constraint shall be treated in the $k$-th iteration.
We follow~\cite{BB97} and assume that $r$ is a \textit{random mapping}, i.e. each value in the index set $I$ is taken infinitely often.
Note that random mappings in this sense are not necessarily stochastic objects; a cyclic control sequence $r(k) = (k\mod m)+1$ is also random in this sense.
Algorithm~\ref{alg:BPSFP} (BPSFP) then generates a sequence of iterates $x_k$ similarly to the method of \emph{random Bregman projections} from~\cite{BB97}.

\begin{algorithm}[h]
  \caption{Bregman projections for split feasibility problems (BPSFP)}
  \label{alg:BPSFP}
  \begin{algorithmic}[1]
    \REQUIRE{starting points $x_0\in\RR^n$ and $x_0^* \in \partial
    f(x_0)$, a control sequence $r:\NN\to I$}
    \ENSURE{a feasible point for~\eqref{eq:SFP}}
    \STATE initialize $k =  0$
    \REPEAT
    \IF[\texttt{simple constraint $x\in C_{r(k)}$}]{$r(k)\in I_C$}
    \STATE update primal variable $x_k = \Pi_{C_{r(k)}}^{x^*_{k-1}}(x_{k-1})$
    \STATE choose dual variable $x_k^* \in \partial f(x_k)$ admissible for $x_k$
    (cf. Lemma~\ref{lem:Pi})
    \ELSIF[\texttt{difficult constraint $A_{r(k)}x\in Q_{r(k)}$}]{$r(k)\in I_Q$} 
    \STATE calculate $w_k = A_{r(k)} x_{k-1} - P_{Q_{r(k)}}\big(A_{r(k)} x_{k-1}\big)$\label{algline:wk}
    \STATE calculate $\beta_k = \scp{A_{r(k)}^T w_k}{x_{k-1}} - \norm[2]{w_k}^2$\label{algline:betak}
    \STATE update primal variable $x_k = \Pi_{H_{\le}(A_{r(k)}^T w_k, \beta_k)}^{x^*_{k-1}}(x_{k-1})$ (cf. Definition~\ref{def:aff-space-half-space})\label{algline:breg-proj-halfspace}
    \STATE choose dual variable $x_k^* \in \partial f(x_k)$ admissible for $x_k$
    (cf. Lemma~\ref{lem:Pi})
    \ENDIF
    \STATE increment $k =  k+1$
    \UNTIL{a stopping criterion is satisfied}
  \end{algorithmic}
\end{algorithm}

\begin{thm} \label{thm:convergence}
The sequence $(x_k)_k$ produced by algorithm~\ref{alg:BPSFP} converges to a solution $\hat{x}$ of~\eqref{eq:SFP} and the sequence $(x^*_k)_k$ converges to a subgradient $x^* \in \partial f(\hat{x})$.
For all solutions $x$ of~\eqref{eq:SFP} we have
\be
D^{x_k^*}(x_k,x) \le D^{x_{k-1}^*}(x_{k-1},x) - \frac{\alpha}{2} \|x_{k-1}-x_k\|_2^2  \label{eq:decreasing}
\ee
and furthermore, for all $r(k) \in I_Q$,
\be
D^{x_k^*}(x_k,x) \le D^{x_{k-1}^*}(x_{k-1},x) - \frac{\alpha}{2} \frac{\|w_k\|_2^4}{\norm[2]{A_{r(k)}w_k}^2}  \,. \label{eq:decreasing_Q}
\ee
\end{thm}
\begin{proof}
The proof is similar to the ones in~\cite{BB97, SSL08b}.
The inequalities~\eqref{eq:decreasing} and~\eqref{eq:decreasing_Q} follow directly from~\eqref{eq:Pi-D} and~\eqref{eq:D-decreasing-H} respectively.
More precisely these inequalities hold for all $x \in C_{r(k)}$.
An immediate consequence is that for all $x \in C$ the sequence $\big(D^{x_k^*}(x_k,x)\big)_k$ is decreasing and hence convergent and bounded, and that we have
\be
\lim_{k \to \infty}  \|x_{k-1}-x_k\|_2 = 0 \,, \label{eq:successive-iterates}
\ee
as well as for all $r(k) \in I_Q$
\be
\lim_{k \to \infty}  \|A_{r(k)} x_{k-1} - P_{Q_{r(k)}}\big(A_{r(k)} x_{k-1}\big)\|_2 = \|w_k\|_2 = 0 \,. \label{eq:Q-iterates}
\ee
Coercivity of $f^*$ and boundedness of $\Delta(x_k^*,x)=D^{x_k^*}(x_k,x)$ imply that the sequences $(x^*_k)_k$ is bounded.
Hence $(x^*_k)_k$ has a convergent subsequence $(x^*_{k_l})_l$ with limit $x^* \in \RR^n$.
Since $\nabla f^*$ is continuous we have
\[
x_{k_l}=\nabla f^*(x^*_{k_l}) \to \nabla f^*(x^*)=:\hat{x} \quad \mbox{for} \quad  l \to \infty \,.
\]
We show that $\hat{x}\in C$.
By choosing a subsequence we may, without loss of generality, assume that $r(k_l)=j$ for some fixed $j \in I$ and all $l \in \NN$ and hence $\hat{x} \in C_j$.
In case $j \in I_C$ this is clear since then $x_{k_l} \in C_j$ for all $l \in \NN$.
And in case $j \in I_Q$ this follows from~\eqref{eq:successive-iterates},~\eqref{eq:Q-iterates} and continuity of the orthogonal projection.
Since $r$ is a random mapping we may also assume that $\{r(k_l), r(k_l+1),\ldots,r(k_{l+1}-1)\}=I$.
We define
\[
I_{\mbox{\small{in}}}=\{ i \in I \,|\, \hat{x} \in C_i \} \quad \mbox{and} \quad I_{\mbox{\small{out}}}=I \setminus I_{\mbox{\small{in}}} \,.
\]
and want to show that $I_{\mbox{\small{out}}} = \emptyset$.
Note that $I_{\mbox{\small{in}}} \not= \emptyset$ because $j \in I_{\mbox{\small{in}}}$.
Now assume to the contrary that $I_{\mbox{\small{out}}} \not= \emptyset$.
Then to each $l \in \NN$ there is a maximal Index $m_l \in \{k_l, k_l+1,\ldots,k_{l+1}-2\}$ such that $r(k) \in I_{\mbox{\small{in}}}$ for all $k_l \le k \le m_l$ and $r(m_l+1) \in I_{\mbox{\small{out}}}$.
We remind that~\eqref{eq:decreasing} holds for all $x \in C_{r(k)}$ and use the inequality successively to get
\[
D^{x^*_{m_l}}(x_{m_l},\hat{x}) \le D^{x^*_{k_l}}(x_{k_l},\hat{x}) = \Delta(x^*_{k_l},\hat{x}) \quad \mbox{for all} \quad l \in \NN \,.
\]
Since $\Delta$ is continuous the right hand side converges to $\Delta(x^*,\hat{x})=0$ for $l \to  \infty$, see~\eqref{eq:Delta-subgradients}.
Hence also the left hand side converges to zero.
By lemma~\ref{lem:D} this implies $\lim_{l \to \infty} x_{m_l}=\hat{x}$.
By passing to a subsequence we may assume $r(m_l+1)=j_{\mbox{\small{out}}} \in I_{\mbox{\small{out}}}$ for all $l \in \NN$ and that the sequence $(x_{m_l+1})_l$ converges to some $\tilde{x} \in C_{j_{\mbox{\small{out}}}}$.
From~\eqref{eq:successive-iterates} we get
\[
\|\hat{x}-\tilde{x}\|_2=\lim_{l \to \infty}  \|x_{m_l}-x_{m_l+1}\|_2 = 0 \,,
\]
and conclude that $\hat{x} = \tilde{x} \in C_{j_{\mbox{\small{out}}}}$, i.e. $j_{\mbox{\small{out}}} \in I_{\mbox{\small{in}}}$ which is a contradiction.
Hence we indeed have $\hat{x}\in C$.
Since the sequence $\big(D^{x_k^*}(x_k,\hat{x})\big)_k$ converges and the subsequence $\big(D^{x^*_{k_l}}(x_{k_l},\hat{x})\big)_l$ converges to zero, the whole sequence $\big(D^{x_k^*}(x_k,\hat{x})\big)_k$ must converge to zero as well.
By lemma~\ref{lem:D} we conclude that $\lim_{k \to \infty} x_k=\hat{x}$.
\end{proof}

According to lemma~\ref{lem:Pi_affine_space_hyperplane}~(b) the computation of the Bregman projection $\Pi_{H_k}^{x^*_{k-1}}(x_{k-1})$ onto the halfspace $H_k$ in step~\ref{algline:breg-proj-halfspace} of algorithm~\ref{alg:BPSFP} amounts to an exact solution of the minimization problem~\eqref{eq:exact-linesearch-BP-halfspace}.
In practice, this is feasible only in special cases, e.g. for $f(x)=\|x\|_2^2$ or $f(x)=\lambda \|x\|_1 + \frac{1}{2} \|x\|_2^2$, as we will see below, but in general one must resort to inexact steps.
Fortunately the assertions of theorem~\ref{thm:convergence} remain true for several of such inexact linesearches as well.

\begin{thm} \label{thm:inexact}
Let $c_1,c_2 >0$ and for a given $x_{k-1}^*$ and $w_k$ and $\beta_k$ according to lines~\ref{algline:wk} and~\ref{algline:betak} of algorithm~\ref{alg:BPSFP}, respectively, set
\[
g(t):=f^*(x^*_{k-1} - t \cdot A_{r(k)}^T w_k) + t\cdot \beta_k.
\]
\begin{itemize}
\item[(a)] If $t_k$ is chosen such that
  \begin{enumerate}
  \item[i)] $t_k\geq c_1$
  \item[ii)] $g'(t_k) \le  0 $
  \item[iii)] $g(t_k) \le  g(0) + c_2 \cdot t_k \cdot g'(0)$
  \end{enumerate}
  and $x_k$ and $x_k^*$ are updated according to 
  \[
  x^*_k := x^*_{k-1} - t_k \cdot A_{r(k)}^T w_k \quad \mbox{and} \quad x_k:=\nabla f^*(x^*_k),
  \]
  then the assertions of theorem~\ref{thm:convergence} remain true.
\item[(b)] The stepsizes
  \[
  \tilde{t}_k := \alpha \cdot \frac{\norm[2]{w_k}^2}{\norm[2]{A_{r(k)}^T w_k}^2} \quad\mbox{and}\quad \bar t_k :=  \frac{\alpha}{\norm[2]{A_{r(k)}^T}^2}\,. \label{eq:tilde_t_k}
  \]
  both fulfill conditions i)--iii) with $c_1 =
  \alpha/\norm[2]{A_{r(k)}^T}^2$ and $c_2 = 1/2$.
\end{itemize}
\end{thm}

\begin{proof}
(a) Condition ii) ensures that $x_k$ and $x^*_k$ fulfill the variational inequality~\eqref{eq:Pi-variational} for all $x \in C_{r(k)} \subset H_k$, since for all such $x$ we have
\bean
\scp{x^*_k-x^*_{k-1}}{x-x_k} &=& t_k \cdot \big(\scp{w_k}{A_{r(k)}x_k} - \scp{w_k}{A_{r(k)}x}\big)  \\
&\ge &  t_k \cdot \big(\scp{w_k}{A_{r(k)}x_k} - \beta_k\big) \\
&=& - t_k \cdot g'(t_k) \ge 0 \,.
\eean
Hence the equivalent inequality~\eqref{eq:decreasing} holds as well.
Condition iii) ensures that we have
\[
D^{x_k^*}(x_k,x) \le D^{x_{k-1}^*}(x_{k-1},x) - c_2 \cdot t_k \cdot \|w_k\|_2^2
\]
for all $x \in C_{r(k)} \subset H_k$.
Like~\eqref{eq:decreasing_Q} and due to i) this forces $\|w_k\|_2$ to converge to zero.
Together with~\eqref{eq:decreasing} this suffices to show convergence as in the proof of theorem~\ref{thm:convergence}.

(b) Clearly, both $\tilde t_k$ and $\bar t_k$ fulfill condition i).
To see that the condition ii) is fulfilled observe that due to the Lipschitz continuity of $\nabla f^*$~\eqref{eq:Lipschitz} we have for all $t \ge 0$
\bea
g'(t) &=& - \scp{A_{r(k)}^T w_k}{\nabla f^*(x^*_{k-1} - t \cdot A_{r(k)}^T w_k)} + \beta_k \nonumber\\
&=& \scp{A_{r(k)}^T w_k}{\nabla f^*(x^*_{k-1}) - \nabla f^*(x^*_{k-1} - t\cdot A_{r(k)}^T w_k)} - \|w_k\|_2^2 \nonumber\\
& \le & \frac{t}{\alpha} \cdot \|A_{r(k)}^T w_k\|_2^2 - \|w_k\|_2^2.  \label{eq:aux-inexact}
\eea
Obviously, it then holds that $g'(\tilde t_k)\leq 0$ and $g'(\bar t_k)\leq 0$.
For condition iii) we use~\eqref{eq:remainder_inequality} to get
\bean
g(t) &\leq & g(0) + t g'(0) + \tfrac{t^2}{2\alpha} \norm[2]{A_{r(k)}^T w_k}^2 \\
&\leq & g(0) + t g'(0) + \tfrac{t^2}{2\alpha} \norm[2]{A_{r(k)}^T}^2 \cdot \norm[2]{w_k}^2 \,.
\eean
Together with $g'(0)= -\norm[2]{w_k}^2$ we infer that iii) holds with $c_2 = 1/2$ both for $\tilde{t}_k$ and $\bar t_k$.
\end{proof}

Based on theorem~\ref{thm:inexact} we obtain three adaptions of algorithm~\ref{alg:BPSFP}:
\begin{enumerate}
\item Two variants of BPSFP with dynamic stepsize (algorithm~\ref{alg:BPSFPdynamic}) which use $\tilde t_k$ or $\bar t_k$ from theorem~\ref{thm:inexact}.
\item The BPSFP with inexact linesearch which increases the stepsize as long as the condition $g'(t_k)\leq 0$ is fulfilled (algorithm~\ref{alg:BPSFPinexact}).
\end{enumerate}

\begin{algorithm}[htb]
  \caption{Bregman projections for split feasibility problems (BPSFP) with dynamic or constant step-size}
  \label{alg:BPSFPdynamic}
  \begin{algorithmic}[1]
    \REQUIRE{starting points $x_0\in\RR^n$ and $x_0^* \in \partial
    f(x_0)$, modulus of strong convexity $\alpha>0$, a control sequence $r:\NN\to I$}
    \ENSURE{a feasible point for~\eqref{eq:SFP}}
    \STATE initialize $k =  0$
    \REPEAT
    \IF[\texttt{simple constraint $x\in C_{r(k)}$}]{$r(k)\in I_C$}
    \STATE update primal variable $x_k = \Pi_{C_{r(k)}}^{x^*_{k-1}}(x_{k-1})$
    \STATE choose dual variable $x_k^* \in \partial f(x_k)$ admissible for $x_k$
    (cf. Lemma~\ref{lem:Pi})
    \ELSIF[\texttt{difficult constraint $A_{r(k)}x\in Q_{r(k)}$}]{$r(k)\in I_Q$} 
    \STATE calculate $w_k = A_{r(k)} x_{k-1} - P_{Q_{r(k)}}\big(A_{r(k)} x_{k-1}\big)$
    \STATE calculate either $t_k = \alpha\tfrac{\norm[2]{w_k}^2}{\norm[2]{A_{r(k)}^Tw_k}^2}$ or $t_k = \tfrac{\alpha}{\norm[2]{A_{r(k)}^T}^2}$
    \STATE update dual variable $x_k^* = x_{k-1}^* - t_k A_{r(k)}^Tw_k$
    \STATE update primal variable $x_k = \nabla f^*(x^*_k)$
    \ENDIF
    \STATE increment $k =  k+1$
    \UNTIL{a stopping criterion is satisfied}
  \end{algorithmic}
\end{algorithm}

\begin{algorithm}[htb]
  \caption{Bregman projections for split feasibility problems (BPSFP) with inexact linesearch}
  \label{alg:BPSFPinexact}
  \begin{algorithmic}[1]
    \REQUIRE{starting points $x_0\in\RR^n$ and $x_0^* \in \partial
    f(x_0)$, modulus of strong convexity $\alpha>0$, a control sequence $r:\NN\to I$ and a constant $c>1$}
    \ENSURE{a feasible point for~\eqref{eq:SFP}}
    \STATE initialize $k =  0$
    \REPEAT
    \IF[\texttt{simple constraint $x\in C_{r(k)}$}]{$r(k)\in I_C$}
    \STATE update primal variable $x_k = \Pi_{C_{r(k)}}^{x^*_{k-1}}(x_{k-1})$
    \STATE choose dual variable $x_k^* \in \partial f(x_k)$ admissible for $x_k$
    (cf. Lemma~\ref{lem:Pi})
    \ELSIF[\texttt{difficult constraint $A_{r(k)}x\in Q_{r(k)}$}]{$r(k)\in I_Q$} 
    \STATE calculate $w_k = A_{r(k)} x_{k-1} - P_{Q_{r(k)}}\big(A_{r(k)} x_{k-1}\big)$
    \STATE calculate $\beta_k = \scp{A_{r(k)}^Tw_k}{x_{k-1}} - \norm[2]{w_k}^2$
    \STATE calculate $t_k = \alpha\tfrac{\norm[2]{w_k}^2}{\norm[2]{A_{r(k)}^Tw_k}^2}$
    \STATE choose $p\in \NN$ the largest integer such that\\
    $\beta_k\leq \scp{A_{r(k)}^Tw_k}{\nabla f^*(x_{k-1}^* - c^p\tilde t_k A_{r(k)}^T w_k}$
    \STATE set $t_k = c^p \tilde t_k$
    \STATE update dual variable $x_k^* = x_{k-1}^* - t_k A_{r(k)}^Tw_k$
    \STATE update primal variable $x_k = \nabla f^*(x^*_k)$
    \ENDIF
    \STATE increment $k =  k+1$
    \UNTIL{a stopping criterion is satisfied}
  \end{algorithmic}
\end{algorithm}

\subsection{Minimization problems with equality constraints}
\label{sec:meth-equal-constraints}

Algorithm~\ref{alg:BPSFP} and its variants with inexact stepsizes solve~\eqref{eq:SFP}.
But does this also allow to compute solutions to the optimization problem~\eqref{eq:f-sfp}?
It turns out that the answer is ``yes'' for linear constraints
\be
\min_{x \in \RR^n} f(x) \quad \mbox{s.t.} \quad Ax=b \label{eq:CP}
\ee
and several different algorithmic formulations.

\begin{cor}
  \label{cor:conv-lin-eq}
  Let $I_1,\dots, I_l$ be a partition of $\{1,\dots,m\}$, denote by $A_j$ the matrix consisting of the rows of $A$ indexed by $I_j$ and let $b_j$ denote the vector consisting of the entries of $b$ indexed by $I_j$.
  The constraints $A_jx=b_j$ may be considered both as simple constraints $C_j = L(A_j,b_j)$ (cf. lemma~\ref{lem:Pi_affine_space_hyperplane} (b)) or as difficult constraints $C_{Q_j}=\set{x \in \RR^n}{Ax_j \in Q_j=\{b_j\}}$.
  Then algorithm~\ref{alg:BPSFP} (or its variants) with Bregman projections with respect to $f$ and initialized with $x_0^* \in \Rcal(A^T)$ and $x_0=\nabla f^*(x_0^*)$ converges to a solution of~\eqref{eq:CP}.
\end{cor}

\begin{proof}
	Since $x_0^* \in \Rcal(A^T)$ and the updates are of the form $x_k^* = x_{k-1}^* - A^T v_k$ for some $v_k \in \RR^m$, we get $x^*=\lim_{k \to \infty} x_k^* \in \Rcal(A^T)$.
	Hence $\hat{x}=\lim_{k \to \infty} x_k$ fulfills the optimality conditions $A\hat{x}=b$ and $\partial f(\hat x) \cup \Rcal(A^T) \not= \emptyset$ for~\eqref{eq:CP}.
\end{proof}

\subsection{Examples}
\label{sec:examples}

In this section we will see that one can recover several well-known methods and also formulate new ones.

\subsubsection{Minimal error method, Landweber iteration, Kaczmarz's method}
\label{sec:me_l_k}

Consider the linearly constrained minimization problem~\eqref{eq:CP} with $f(x) = \frac{1}{2}\norm[2]{x}^2$.
In this case we have $f^*=f$ and $\nabla f^*(x)=x$ and Bregman projections with respect to $f$ are just orthogonal projections.\\

If we apply algorithm~\ref{alg:BPSFP} with just one constraint $Ax\in Q = \{b\}$ and exact linesearch then a simple calculation shows that the iteration reads as
\ben
x_{k}  =  x_{k-1} - \frac{\|Ax_{k-1} - b\|_2^2}{\|A^T(Ax_{k-1} - b)\|_2^2} \cdot A^T(Ax_{k-1} - b) \,.
\een
This is the so-called \emph{minimal error method}~\cite[Section 3.4]{KNS08}.
Note that in this case the exact stepsize $t_k=\frac{\|Ax_{k-1} - b\|_2^2}{\|A^T(Ax_{k-1} - b)\|_2^2}$ coincides with the ``dynamic stepsize'' $\tilde{t}_k$ from algorithm~\ref{alg:BPSFPdynamic}.
Using a fixed stepsize $t_k = 1/\norm[2]{A}^2$ leads to the well known Landweber method~\cite{Lan51}.\\

Instead we can also regard the constraint $Ax=b$ as an intersection of several ``smaller'' linear constraints as in corollary~\ref{cor:conv-lin-eq}.
Using just orthogonal projections onto the hyperplanes $H(a_i,b_i)$, where by $a_i$ we denote the $i$-th row of $A$, the resulting iteration is
\ben
   x_{k} = x_{k-1} - \frac{\scp{a_{r(k)}}{x_{k-1}} - b_{r(k)}}{\norm[2]{a_{r(k)}}^2}a_{r(k)}^T \,.
\een
This is known as the Kaczmarz method~\cite{K37} and also under the name ``algebraic reconstruction technique'' in the tomography community~\cite{GBH70}.

\subsubsection{Linearized Bregman method and ``sparse'' Kaczmarz}
\label{sec:linbreg}
We also recover the linearized Bregman method.
To that end, consider~\eqref{eq:CP} with $f(x) = \lambda\norm[1]{x} + \tfrac{1}{2}\norm[2]{x}^2$.
Its subgradient is given by
\ben
\partial f(x)_i =
\begin{cases}
  \lambda\sign(x_i) + x_i &, x_i\neq 0\\
  [-\lambda,\lambda] &, x_i = 0
\end{cases}.
\een
By subgradient inversion we get $\nabla f^*(x^*) = S_{\lambda}(x^*)$ and $f^*(x^*) = \tfrac{1}{2}\norm[2]{S_{\lambda}(x^*)}^2$.\\

Applying algorithm~\ref{alg:BPSFP} with single constraint $Ax=b$ then leads to the iteration
\bean
x_k^* &=& x_{k-1}^* - t_k A^T(Ax_{k-1} - b) \\
x_k &=& S_{\lambda}(x_{k}^*)
\eean
which is, up to the stepsize $t_k$, precisely the linearized Bregman method.
Note that not only the constant stepsize $t_k = 1/\norm[2]{A}^2$ is allowed, but also the dynamic stepsize $\tilde t_k = \norm[2]{Ax_{k-1} - b}^2/\norm[2]{A^T(Ax_{k-1}-b)}^2$ and the inexact stepsize according to algorithm~\ref{alg:BPSFPinexact} lead to convergent methods (cf.~theorem~\ref{thm:inexact}).
Moreover, in this case it is also tractable to use an exact linesearch according to lemma~\ref{lem:Pi_affine_space_hyperplane} (b), i.e. to solve a minimization problem of the form
\ben
\min_{t \in \RR} g(t)= \tfrac{1}{2}\norm[2]{S_{\lambda}(x^*-t\, a)}^2+ t \, \beta  \,.
\een 
Since in this case $g(t)$ is piecewise quadratic with piecewise linear derivative $g'(t)$ we can even perform an exact linesearch as follows:
At first determine the kinks $0 =:t_0 < t_1 < t_2 < \ldots$ and corresponding slopes $s_l$ and intercepts $b_l$ such that
\[
g'(t)=s_l \cdot t + b_l \quad,\quad t \in [t_{l-1},t_l] \,.
\]
Then a point $\hat{t} \in [t_{l-1},t_l]$ minimizes $g(t)$ if and only if 
\[
g'(\hat{t})=0 \quad \Leftrightarrow \quad s_l\not=0, \hat{t}=\frac{-b_l}{s_l}\quad \mbox{or} \quad s_l=0, b_l=0  \,.
\]
Hence it remains to increase $l=1,2,\ldots$ until $\hat{t}:=\frac{-b_l}{s_l} \in [t_{l-1},t_l]$ or $s_l=0, b_l=0$, in which case we may choose $\hat{t}:=t_{l-1}$ as minimizer.

Note that the dynamic and the exact stepsize have an important advantage over the constant stepsize:
The knowledge of the operator norm $\norm[2]{A}$ is not needed.
Moreover, the dynamic stepsize does not involve any new applications of $A$ or $A^T$ as the needed quantities have to be computed anyway.
Also the overhead for the exact stepsize comparably small.
As we will see in section~\ref{sec:step-sizecomparison}, the dynamic and the exact stepsize perform significantly better than the constant stepsize.

Instead of considering $Ax=b$ as a single constraint we can also adopt the idea of the Kaczmarz method.
Using just the hyperplanes $H(a_i,b_i)$ as in example~\ref{sec:me_l_k} then leads to an iteration of the form
\bean
x_k^* &=& x_{k-1}^* - t_k \cdot a_{r(k)}^T  \\
x_k &=& S_{\lambda}(x_{k}^*) \,,
\eean
and due to the thresholding operation $S_{\lambda}$ we end up with a \emph{``sparse'' Kaczmarz method}.
A different approach to sparse solutions of \emph{over}determined systems by a Kaczmarz method has been proposed recently in~\cite{MY13}.

\subsubsection{Linearized Bregman method for different noise-models}
\label{sec:linbreg-non-gauss}

To tackle the presence of noise in the right hand side of the linear constraint $Ax=b$, i.e. only $b^\delta$ with $\norm[]{b-b^\delta}\le \delta$ is given, we consider the problem
\be
\min_{x \in \RR^n} f(x) \quad \mbox{s.t.} \quad \norm{Ax-b^\delta} \le \delta \,. \label{eq:CPnoisy}
\ee
The choice of the norm in which the constraint is formulated is dictated by the noise characteristics, e.g. the $2$-norm is appropriate for Gaussian noise, the $1$-norm for impulsive noise, and the $\infty$-norm for uniformly distributed noise.
Applying algorithm~\ref{alg:BPSFP} by considering the constraint as
\ben
Ax \in Q = \{y \in \RR^m \:|\: \norm[]{y-b^\delta}\le \delta\}
\een
leads to the simple iteration
\be
\label{eq:BPSFP-nongauss}
\begin{split}
  x_{k}^* &= x_{k-1}^* - t_{k} A^T(Ax_{k-1} - P_Q(Ax_{k-1}))\\
  x_{k} &= S_{\lambda}(x_{k}^*) \,.
\end{split}
\ee
The projections $P_Q$ only involve orthogonal projections onto the respective norm-balls.
This is easy for the $2$-norm,
\ben
Ax - P_Q(Ax) = \max\{0,1-\tfrac{\delta}{\norm[2]{Ax - b^\delta}}\} \cdot (Ax - b^\delta) \,,
\een
and the $\infty$-norm,
\ben
Ax - P_Q(Ax) = S_\delta(Ax - b^\delta) \,,
\een
and fast algorithms are available for projections onto the 1-norm-ball~\cite{DSSC08}.
The iterates $x_k$ are guaranteed to converge to a feasible point, however there is no guarantee that the limit point will be optimal for~\eqref{eq:CPnoisy}.
Nevertheless, as we will see in section~\ref{sec:sparse-rec-non-gaussian} where numerical experiments are reported, the method sometimes returns optimal solutions.

\section{Numerical experiments}
\label{sec:num-experiments}

\subsection{Comparison of stepsizes for the linearized Bregman method}
\label{sec:step-sizecomparison}

We conduct an experiment on the comparison of step-size rules.
We consider the problem of finding sparse solution of equations, i.e. precisely problem~\eqref{eq:sparse-lin-sys}.
For a given matrix $A$ we form recoverable sparse vector $x^\dag$ using L1TestPack~\cite{lorenz2013bpdn}, compute the corresponding right hand side $b=Ax^\dag$ and choose the regularization parameter $\lambda$ such that the solution of~\eqref{eq:sparse-lin-sys} will return the input (cf.~\cite{Sch12}).
We refer the interested reader to~\cite{Sch12,Yin10,FT07,Don06} for further information about when suffiently sparse solutions can be recoverd exactly by~\eqref{eq:sparse-lin-sys}.
For every instance we compare four different step-size rules:
\begin{enumerate}
\item The constant stepsize $t_k = 1/\norm[2]{A}^2$.
\item The ``dynamic stepsize'' $t_k = \frac{\norm[2]{Ax_{k-1}-b}^2}{\norm[2]{A^T(Ax_{k-1}-b)}^2}$.
\item The exact stepsize as described in example~\ref{sec:linbreg}.
\item The Barzilai-Borwein stepsize with non-monotone linesearch proposed in~\cite{Yin10}\footnote{Code available at \url{http://www.caam.rice.edu/~optimization/linearized_bregman/line_search/lbreg_bbls.html}, accessed August 27th, 2013.}.
\end{enumerate}
The results are summarized in figures~\ref{fig:stepsize-1}--~\ref{fig:stepsize-3}. 
Note that the constant stepsize is always the slowest (however, in figure~\ref{fig:stepsize-3}, the dynamic stepsize is basically equal, due to the structure of the matrix).
The exact stepsize performs better than the dynamic stepsize and even outperform the Barzilai-Borwein stepsize with non-monotone linesearch (cf.~figure~\ref{fig:stepsize-2}).

\begin{figure}[htb]
  \centering
  \begin{tikzpicture}[yscale=0.16,xscale=0.012]
    \draw[->] (0,-10) -- (480,-10)node[below]{$k$};
    \draw[->] (0,-10) -- (0,2)node[left]{\tiny $\norm[2]{Ax_k-b}$};
    \foreach \i in {-8,-4,0}{
      \draw (3,\i) -- (-3,\i)node[left]{\tiny$10^{\i}$};
    }
    \foreach \i in {100,200,300,400}{
      \draw (\i,-9.5) -- (\i,-10.5)node[below]{\tiny$\i$};
    }
    
    \draw[blue!80!black,dashed] plot file {data/linbreg_fixed_resgauss_gauss1000x2000_s60.dat};
    \draw[red!80!black,dotted] plot file {data/linbreg_bbls_resgauss_gauss1000x2000_s60.dat};
    \draw[green!60!black] plot file {data/linbreg_exact_resgauss_gauss1000x2000_s60.dat};
    \draw[white!50!black] plot file {data/linbreg_dyn_resgauss_gauss1000x2000_s60.dat};

    \draw[blue!80!black,dashed,thick] (0,-13.5) -- (50,-13.5)node[right]{\scriptsize constant stepsize};
    \draw[white!50!black,thick] (0,-15) -- (50,-15)node[right]{\scriptsize dynamic};
    \draw[green!60!black,thick] (300,-13.5) -- (350,-13.5)node[right]{\scriptsize exact};
    \draw[red!80!black,dotted,thick] (300,-15) -- (350,-15)node[right]{\scriptsize bbls}; 
    
  \end{tikzpicture}

  \caption{Stepsize comparison. $A\in\RR^{1000\times 2000}$ random Gaussian matrix, $x^\dag$ with 60 non-zero entries, random Gaussian entries.}
  \label{fig:stepsize-1}
\end{figure}
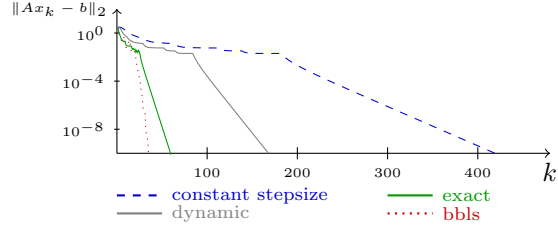

\begin{figure}[htb]
  \centering
  \begin{tikzpicture}[yscale=0.16,xscale=0.012]
    \draw[->] (0,-10) -- (540,-10)node[below]{$k$};
    \draw[->] (0,-10) -- (0,2)node[left]{\tiny $\norm[2]{Ax_k-b}$};
    \foreach \i in {-8,-4,0}{
      \draw (3,\i) -- (-3,\i)node[left]{\tiny$10^{\i}$};
    }
    \foreach \i in {100,200,300,400,500}{
      \draw (\i,-9.5) -- (\i,-10.5)node[below]{\tiny$\i$};
    }
    
    \draw[blue!80!black,dashed] plot file {data/linbreg_fixed_resbernoulli_bernoulli_2000x6000_s60.dat};
    \draw[red!80!black,dotted] plot file {data/linbreg_bbls_resbernoulli_bernoulli_2000x6000_s60.dat};
    \draw[green!60!black] plot file {data/linbreg_exact_resbernoulli_bernoulli_2000x6000_s60.dat};
    \draw[white!50!black] plot file {data/linbreg_dyn_resbernoulli_bernoulli_2000x6000_s60.dat};

    \draw[blue!80!black,dashed,thick] (0,-13.5) -- (50,-13.5)node[right]{\scriptsize constant stepsize};
    \draw[white!50!black,thick] (0,-15) -- (50,-15)node[right]{\scriptsize dynamic};
    \draw[green!60!black,thick] (300,-13.5) -- (350,-13.5)node[right]{\scriptsize exact};
    \draw[red!80!black,dotted,thick] (300,-15) -- (350,-15)node[right]{\scriptsize bbls}; 
      \end{tikzpicture}

  \caption{Stepsize comparison. $A\in\RR^{2000\times 6000}$ random Bernoulli matrix, $x^\dag$ with 60 non-zero entries, random Bernoulli entries.}
  \label{fig:stepsize-2}
\end{figure}
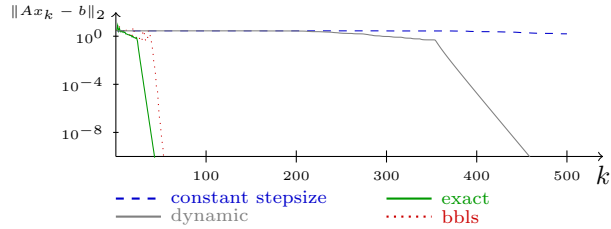

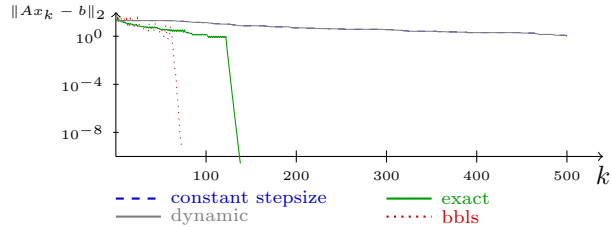
\begin{figure}[htb]
  \centering
  \begin{tikzpicture}[yscale=0.16,xscale=0.012]
    \draw[->] (0,-10) -- (540,-10)node[below]{$k$};
    \draw[->] (0,-10) -- (0,2)node[left]{\tiny $\norm[2]{Ax_k-b}$};
    \foreach \i in {-8,-4,0}{
      \draw (3,\i) -- (-3,\i)node[left]{\tiny$10^{\i}$};
    }
    \foreach \i in {100,200,300,400,500}{
      \draw (\i,-9.5) -- (\i,-10.5)node[below]{\tiny$\i$};
    }
    
    \draw[blue!80!black,dashed] plot file {data/linbreg_fixed_respartdct_dynamic1_2000x6000_s50.dat};
    \draw[red!80!black,dotted] plot file {data/linbreg_bbls_respartdct_dynamic1_2000x6000_s50.dat};
    \draw[green!60!black] plot file {data/linbreg_exact_respartdct_dynamic1_2000x6000_s50.dat};
    \draw[white!50!black] plot file {data/linbreg_dyn_respartdct_dynamic1_2000x6000_s50.dat};
    
    \draw[blue!80!black,dashed,thick] (0,-13.5) -- (50,-13.5)node[right]{\scriptsize constant stepsize};
    \draw[white!50!black,thick] (0,-15) -- (50,-15)node[right]{\scriptsize dynamic};
    \draw[green!60!black,thick] (300,-13.5) -- (350,-13.5)node[right]{\scriptsize exact};
    \draw[red!80!black,dotted,thick] (300,-15) -- (350,-15)node[right]{\scriptsize bbls}; 
    
  \end{tikzpicture}

  \caption{Stepsize comparison. $A\in\RR^{2000\times 6000}$ random partial DCT matrix, $x^\dag$ with 50 non-zero entries, randomly with large dynamic range.}
  \label{fig:stepsize-3}
\end{figure}

\subsection{Sparse recovery with non-Gaussian noise}
\label{sec:sparse-rec-non-gaussian}

To illustrate the performance of the BPSFP algorithms from example~\ref{sec:linbreg-non-gauss}, we consider the problem of recovering sparse solutions of linear equations $Ax=b$, where only noisy data $b^\delta$ is available for different noise-models.

\subsubsection{Impulsive noise}
\label{sec:impuls-noise}

For some matrix $A\in\RR^{1000\times 2000}$ we produce a sparse vector $x^\dag\in\RR^{2000}$ with only 30 non-zero entries and calculate $b = Ax^\dag$.
To obtain $b^\delta$,  100 randomly chosen entries of $b$ are changed to either the value of the largest or smallest entry of $b$ (with equal probability), see figure~\ref{fig:data-impulse-noise}.
We compute $\delta = \norm[1]{b-b^\delta}$ and consider
\be
\label{eq:sparse-impulse-noise}
\min_{x \in \RR^n} \lambda \norm[1]{x} + \tfrac{1}{2}\norm[2]{x}^2\quad\mbox{s.t.}\quad \norm[1]{Ax-b^\delta}\leq \delta \,.
\ee
To obtain an optimal solution we use the primal-dual method from~\cite{CP11}.
The method is based on the saddle point formulation of~\eqref{eq:sparse-impulse-noise}, which in this case reads as
\ben
\min_{x \in \RR^n} \max_{y \in \RR^m} \lambda \norm[1]{x} + \tfrac{1}{2}\norm[2]{x}^2 + \scp{Ax}{y} - \delta\norm[\infty]{y} - \scp{b^\delta}{y} \,.
\een
With $F(x) = \lambda \norm[1]{x} + \tfrac{1}{2}\norm[2]{x}^2$ and $G(y) = \delta\norm[\infty]{y} + \scp{b^\delta}{y}$ one then iterates
\ben
\begin{split}
  x_k & = \prox_{\tau F}(x_{k-1} - \tau A^Ty_{k-1})\\
  y_k & = \prox_{\sigma G}(y_{k-1} + \sigma A(2x_k - x_{k-1})) \,.
\end{split}
\een
In~\cite{CP11} the algorithm is shown to converge to a solution of~\eqref{eq:sparse-impulse-noise} if $\tau\sigma< \norm[2]{A}^{-2}$.
Note that since $\prox_{\sigma G}(y) = y - \sigma\prox_{\sigma^{-1}G^*}(\sigma^{-1}y)$ and
\[
G^*(y) =
\begin{cases}
  0 &,\quad \norm[1]{y-b}\leq\delta\\
  \infty &,\quad \mbox{else}
\end{cases} \,,
\]
we can evaluate $\prox_{\sigma G}$ also by projecting onto the 1-norm-ball (using the method from~\cite{DSSC08}).
The performance of the primal-dual method and the variants of BPSFP is shown in figure~\ref{fig:algorithms-impulse-noise}.
Remarkably the BPSFP with exact linesearch does not only produce an optimal solution of~\eqref{eq:sparse-impulse-noise}, but also does this very fast.
By contrast the dynamic stepsize shows the well know behavior of linearized Bregman methods that the iterates tend to stagnate from time to time.
In fact the algorithm also converges to an optimal solution with high accuracy but needs about 1200 iterations.
We note that in this case, in spite of the noise, the optimal solution is \emph{equal} to the original $x^\dag$ (up to a numerical precision).
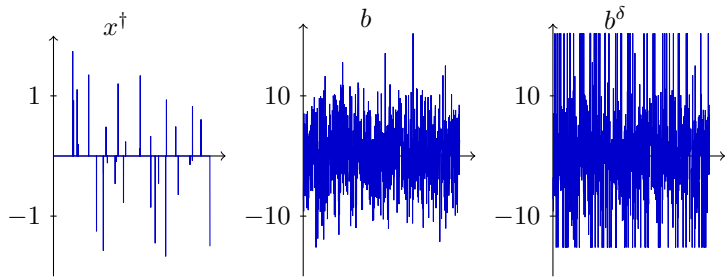
\begin{figure}[bht]
  \centering
  \begin{tikzpicture}[yscale=0.8,xscale=2.08]
    \draw[->] (0,0) -- (1.1,0);
    \draw[->] (0,-2) -- (0,2);
    \foreach \i in {-1,1}{
      \draw (0.02,\i) -- (-0.02,\i)node[left]{$\i$};
    }
    
    \draw[blue!80!black] plot file {data/impulsive_noise/x.dat};
    \draw (0.4,2.2)node{$x^\dag$};
  \end{tikzpicture}
  \begin{tikzpicture}[yscale=0.08,xscale=2.08]
    \draw[->] (0,0) -- (1.1,0);
    \draw[->] (0,-20) -- (0,22);
    \foreach \i in {-10,10}{
      \draw (0.02,\i) -- (-0.02,\i)node[left]{$\i$};
    }
    
    \draw[blue!80!black] plot file {data/impulsive_noise/bdag.dat};
    \draw (0.4,23)node{$b$};
  \end{tikzpicture}
  \begin{tikzpicture}[yscale=0.08,xscale=2.08]
    \draw[->] (0,0) -- (1.1,0);
    \draw[->] (0,-20) -- (0,22);
    \foreach \i in {-10,10}{
      \draw (0.02,\i) -- (-0.02,\i)node[left]{$\i$};
    }
    
    \draw[blue!80!black] plot file {data/impulsive_noise/bdelta.dat};
    \draw (0.4,23)node{$b^\delta$};
  \end{tikzpicture}
  \caption{Data for the experiment of sparse recovery with impulsive noise. Left: sparse vector $x^\dag$, middle: right hand side $b = Ax^\dag$, right: noisy data $b^\delta$, degraded by impulsive noise.}
  \label{fig:data-impulse-noise}
\end{figure}

\begin{figure}[bht]
  \centering
  \begin{tikzpicture}[yscale=0.04,xscale=0.009]
    \draw[->] (0,0) -- (450,0)node[below]{$k$};
    \draw[->] (0,-1) -- (0,80)node[left]{$F(x_k)$};
    \foreach \i in {20,40,60}{
      \draw (5,\i) -- (-5,\i)node[left]{\scriptsize $\i$};
    }
    \foreach \i in {100,200,300,400}{
      \draw (\i,1) -- (\i,-1)node[below]{\scriptsize $\i$};
    }
    
    \draw[red!80!black,dotted] plot file {data/impulsive_noise/fvalPD.dat};
    \draw[white!50!black] plot file {data/impulsive_noise/fvalBPSFPdyn.dat};
    \draw[green!60!black] plot file {data/impulsive_noise/fvalBPSFPexact.dat};
    
  \end{tikzpicture}
  \begin{tikzpicture}[yscale=0.24,xscale=0.009]
    \draw[->] (0,-8) -- (450,-8)node[below]{$k$};
    \draw[->] (0,-8.5) -- (0,5)node[above]{$\norm[1]{Ax_k-b^\delta}-\delta$};
    \foreach \i in {4,0,-4,-8}{
      \draw (5,\i) -- (-5,\i)node[left]{\scriptsize $\i$};
    }
    \foreach \i in {100,200,300,400}{
      \draw (\i,-7.8) -- (\i,-8.2)node[below]{\scriptsize $\i$};
    }
    
    \begin{scope}
      \clip (0,-8) rectangle (450,5);
      \draw[red!80!black,dotted] plot file {data/impulsive_noise/feasPD.dat};
      \draw[white!50!black] plot file
      {data/impulsive_noise/feasBPSFPdyn.dat}; \draw[green!60!black]
      plot file {data/impulsive_noise/feasBPSFPexact.dat};
    \end{scope}
  \end{tikzpicture}\\
  \begin{tikzpicture}[yscale=0.5,xscale=0.8]
    \draw[red!80!black,dotted,thick] (0,0) --+ (1,0)node[right]{\scriptsize primal-dual}; 
    \draw[green!60!black,thick] (3.5,0) --+ (1,0)node[right]{\scriptsize BPSFP exact};
    \draw[white!50!black,thick] (7,0) --+ (1,0)node[right]{\scriptsize BPSFP dynamic};
    
  \end{tikzpicture}
  \caption{Performance of the different algorithms for problem~\eqref{eq:sparse-impulse-noise}. Left: Objective value $F(x) = \lambda \norm[1]{x} + \tfrac{1}{2}\norm[2]{x}^2$, right: feasibility violation in log-scale, i.e. the value $\norm[1]{Ax_k-b^\delta}-\delta$.}
  \label{fig:algorithms-impulse-noise}
\end{figure}
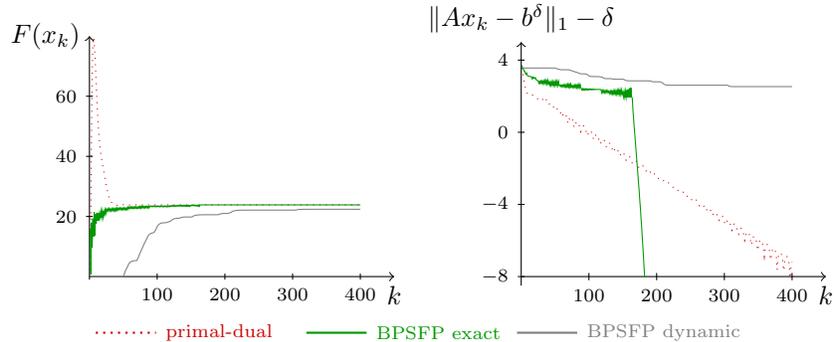

\subsubsection{Uniform noise}
\label{sec:uniform-noise}

In parallel to the previous section we conduct a similar experiment, but now $b^\delta$ is formed by adding uniformly distributed noise with range $[-1,1]$, see figure~\ref{fig:data-uniform-noise}.
We use $\delta = \norm[\infty]{b-b^\delta}$ and consider
\be
\label{eq:sparse-uniform-noise}
\min_{x \in \RR^n} \lambda \norm[1]{x} + \tfrac{1}{2}\norm[2]{x}^2\quad\mbox{s.t.}\quad \norm[\infty]{Ax-b^\delta}\leq \delta \,.
\ee
An optimal solution is again computed with the primal-dual method from~\cite{CP11}.
The performance of the primal-dual method and the variants of BPSFP is shown in Figure~\ref{fig:algorithms-uniform-noise}.
This time the BPSFP algorithms produce feasible but not optimal solutions of~\eqref{eq:sparse-uniform-noise}.
BPSFP with exact stepsize even produces a higher objective value than with the dynamic stepsize.
However, it is remarkable that in terms of the relative reconstruction error, $\text{err}_\text{rel}=\frac{\norm[2]{x-x^\dag}}{\norm[2]{x^\dag}}$, BPSFP with dynamic stepsize produces the lowest value $\text{err}_\text{rel} \approx 0.007$, while the optimal solution leads to $\text{err}_\text{rel} \approx 0.06$ and BPSFP with exact stepsize results in $\text{err}_\text{rel} \approx 0.1$.
Hence the BPSFP methods, although not solving the optimization problem~\eqref{eq:sparse-uniform-noise}, deliver good solutions to the underlying reconstruction problem, cf. Figure~\ref{fig:results-uniform-noise}.
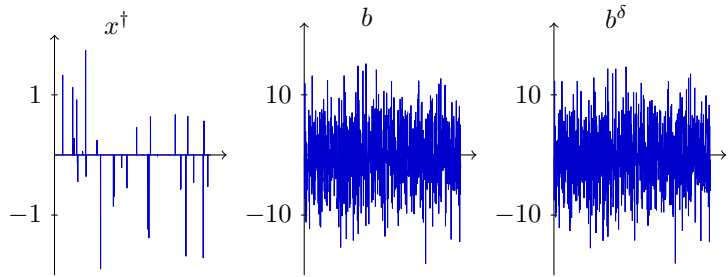
\begin{figure}[bht]
  \centering
  \begin{tikzpicture}[yscale=0.8,xscale=2.08]
    \draw[->] (0,0) -- (1.1,0);
    \draw[->] (0,-2) -- (0,2);
    \foreach \i in {-1,1}{
      \draw (0.02,\i) -- (-0.02,\i)node[left]{$\i$};
    }
    
    \draw[blue!80!black] plot file {data/uniform_noise/x.dat};
    \draw (0.4,2.2)node{$x^\dag$};
  \end{tikzpicture}
  \begin{tikzpicture}[yscale=0.08,xscale=2.08]
    \draw[->] (0,0) -- (1.1,0);
    \draw[->] (0,-20) -- (0,22);
    \foreach \i in {-10,10}{
      \draw (0.02,\i) -- (-0.02,\i)node[left]{$\i$};
    }
    
    \draw[blue!80!black] plot file {data/uniform_noise/bdag.dat};
    \draw (0.4,23)node{$b$};
  \end{tikzpicture}
  \begin{tikzpicture}[yscale=0.08,xscale=2.08]
    \draw[->] (0,0) -- (1.1,0);
    \draw[->] (0,-20) -- (0,22);
    \foreach \i in {-10,10}{
      \draw (0.02,\i) -- (-0.02,\i)node[left]{$\i$};
    }
    
    \draw[blue!80!black] plot file {data/uniform_noise/bdelta.dat};
    \draw (0.4,23)node{$b^\delta$};
  \end{tikzpicture}
  \caption{Data for the experiment of sparse recovery with uniform noise. Left: sparse vector $x^\dag$, middle: right hand side $b = Ax^\dag$, right: noisy data $b^\delta$, degraded by uniform noise.}
  \label{fig:data-uniform-noise}
\end{figure}

\begin{figure}[bth]
  \centering
  \begin{tikzpicture}[yscale=0.04,xscale=0.008]
    \draw[->] (0,0) -- (550,0)node[below]{$k$};
    \draw[->] (0,-1) -- (0,80)node[left]{$F(x_k)$};
    \foreach \i in {20,40,60}{
      \draw (5,\i) -- (-5,\i)node[left]{\scriptsize $\i$};
    }
    \foreach \i/\ii in {125/500,250/1000,375/1500,500/2000}{
      \draw (\i,1) -- (\i,-1)node[below]{\scriptsize $\ii$};
    }
    
    \draw[red!80!black,dotted] plot file {data/uniform_noise/fvalPD.dat};
    \draw[white!50!black] plot file {data/uniform_noise/fvalBPSFPdyn.dat};
    \draw[green!60!black] plot file {data/uniform_noise/fvalBPSFPexact.dat};
    
  \end{tikzpicture}
  \begin{tikzpicture}[yscale=0.32,xscale=0.008]
    \draw[->] (0,-8) -- (550,-8)node[below]{$k$};
    \draw[->] (0,-8.5) -- (0,2)node[above]{$\norm[\infty]{Ax_k-b^\delta}-\delta$};
    \foreach \i in {0,-2,-4,-6,-8}{
      \draw (5,\i) -- (-5,\i)node[left]{\scriptsize $10^{\i}$};
    }
    \foreach \i/\ii in {125/500,250/1000,375/1500,500/2000}{
      \draw (\i,-7.9) -- (\i,-8.1)node[below]{\scriptsize $\ii$};
    }
    
    \begin{scope}
      \clip (0,-8) rectangle (550,5);
      \draw[red!80!black,dotted] plot file {data/uniform_noise/feasPD.dat};
      \draw[white!50!black] plot file
      {data/uniform_noise/feasBPSFPdyn.dat}; \draw[green!60!black]
      plot file {data/uniform_noise/feasBPSFPexact.dat};
    \end{scope}
  \end{tikzpicture}\\
  \begin{tikzpicture}[yscale=0.5,xscale=0.8]
    \draw[red!80!black,dotted,thick] (0,0) --+ (1,0)node[right]{\scriptsize primal-dual}; 
    \draw[green!60!black,thick] (3.5,0) --+ (1,0)node[right]{\scriptsize BPSFP exact};
    \draw[white!50!black,thick] (7,0) --+ (1,0)node[right]{\scriptsize BPSFP dynamic};
    
  \end{tikzpicture}
  \caption{Performance of the different algorithms for problem~\eqref{eq:sparse-uniform-noise}. Left: Objective value $F(x) = \lambda\norm[1]{x} + \tfrac{1}{2}\norm[2]{x}^2$, right: feasibility violation in log-scale, i.e. the value $\norm[\infty]{Ax_k-b^\delta}-\delta$.}
  \label{fig:algorithms-uniform-noise}
\end{figure}
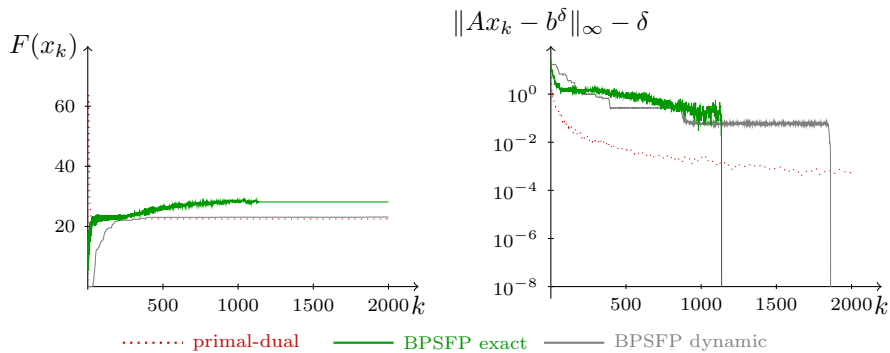

\begin{figure}[htb]
  \centering
  \begin{tabular}{ccc}
    \begin{tikzpicture}[yscale=1,xscale=2.6]
      \draw[->] (0,0) -- (1.1,0);
      \draw[->] (0,-2) -- (0,2);
      \foreach \i in {-1,1}{
        \draw (0.02,\i) -- (-0.02,\i)node[left]{$\i$};
      }
      
      \draw[blue!80!black] plot file {data/uniform_noise/xPD.dat};
    \end{tikzpicture}
    &
    \begin{tikzpicture}[yscale=1,xscale=2.6]
      \draw[->] (0,0) -- (1.1,0);
      \draw[->] (0,-2) -- (0,2);
      \foreach \i in {-1,1}{
        \draw (0.02,\i) -- (-0.02,\i)node[left]{$\i$};
      }
      
      \draw[blue!80!black] plot file {data/uniform_noise/xBPSFPdyn.dat};
    \end{tikzpicture}
    &
    \begin{tikzpicture}[yscale=1,xscale=2.6]
      \draw[->] (0,0) -- (1.1,0);
      \draw[->] (0,-2) -- (0,2);
      \foreach \i in {-1,1}{
        \draw (0.02,\i) -- (-0.02,\i)node[left]{$\i$};
      }
      
      \draw[blue!80!black] plot file {data/uniform_noise/xBPSFPexact.dat};
    \end{tikzpicture}\\
    \scriptsize
    optimal (with primal dual)
    &
    \scriptsize
    BPSFP dynamic
    &
    \scriptsize
    BPSFP exact
  \end{tabular}
  \caption{Reconstructions of $x^\dag$ in the problem of uniform noise of section~\ref{sec:uniform-noise}.}
  \label{fig:results-uniform-noise}
\end{figure}
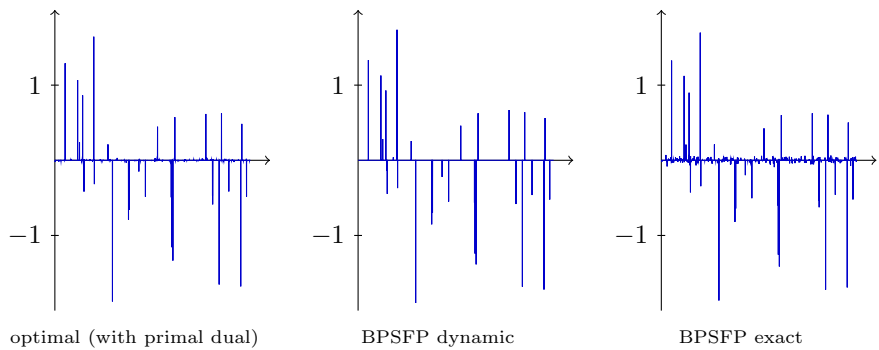

\subsection{Three dimensional reconstruction of planetary nebulae}
\label{sec:nebulae}

Here we describe a real-world, large scale application which can be modelled in the form
\ben
\min_{x\in\RR^n} f(x)\quad\text{s.t.}\quad Ax=b\,,\,x \ge 0 \,,
\een
namely the three dimensional reconstruction of planetary nebulae from a single two dimensional observation.
Planetary nebulae are formed when dying stars eject their matter, forming colorfully glowing gaseous clouds.
They impress and inspire due to their beauty and rich three dimensional structure, and therefore have been studied and catalogued for centuries.
However, due to the large distance of the nebulae to Earth only one single projection of each nebula is, and will be, available.
For the popular educational shows in planetariums, three-dimensional models of nebulae are often created manually, requiring specialized skills and a lot of time.
Automatic reconstruction of three dimensional models of nebulae has been proposed in~\cite{WAGLTM12}.
There a method based on virtual views and tomographic reconstruction has been applied.
However the method requires long computation times.
We follow the more recent proposal in~\cite{WLM13}.
The image of an astronomical nebula can be seen as the integrated emission intensity along the view rays.
In other words the pixel value $s_{i,j}$ of an image of a nebula is formed summing up the intensity values of the sought-after rendering volume $\rho \ge 0$ along one coordinate axis, i.e.
\ben
s_{i,j} = \sum_{k}\rho_{i,j,k}.
\een
We write this with the linear projection operator $P$ as
\ben
s = P\rho.
\een
Among the solutions of this highly underdetermined equation we try to extract the one which fits best to further prior knowledge of the nebula.
Our basic model assumption is that the nebula under consideration obeys some sort of symmetry, e.g. a rotational symmetry which could be guessed for the so-called Butterfly Nebula in figure~\ref{fig:M2-9}.
\begin{figure}[htb]
  \centering
  \includegraphics[width=7cm]{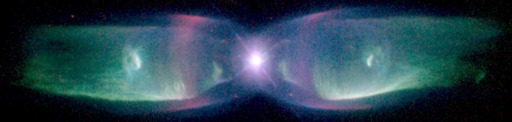}\\\medskip
  \includegraphics[width=7cm]{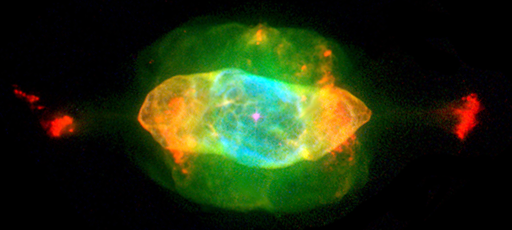}
  
  \caption{Top: The Butterfly Nebula (planetary nebula M2-9), image source: Bruce Balick (University of Washington), Vincent Icke (Leiden University), Garrelt Mellema (Stockholm University), and NASA.
  Bottom: The Saturn Nebula (planetary nebula NGC7009), image source: Bruce Balick (University of Washington), Jason Alexander (University of Washington),
Arsen Hajian (U.S. Naval Observatory),
Yervant Terzian (Cornell University),
Mario Perinotto (University of Florence (Italy)),
Patrizio Patriarchi (Arcetri Observatory (Italy)), and
NASA }
  \label{fig:M2-9}
\end{figure}
Hence we demand that the intensity in the reconstructed volume is close to constant along circles around the symmetry axis.
Moreover the volume is almost empty or, in other words, sparse.
This motivates a mixed $\ell^{1,\infty}$ term: We collect the pixels which belong to a circle around the symmetry axis in a set $G_l$ of indices and form
\ben
f(\rho) = \sum_l |G_l|\max_{(i,j,k)\in G_l}|\rho_{i,j,k}| \,.
\een
Note that the sum is weighted with $|G_l|$, i.e. with the number of pixels in the set $G_l$.
Since $f$ is not strongly convex, we regularize it as proposed at the end of section~\ref{sec:assumptions} and obtain the problem
\be
\label{eq:nebulae-problem}
\min_\rho \lambda f(\rho) + \tfrac{1}{2}\norm[2]{\rho}^2,\quad\text{s.t.}\quad P\rho = s \,,\, \rho \ge 0.
\ee
To employ BPSFP, we incorporate the non-negativity constraint into the function $f$.
Then, the associated proximal mapping amounts to separate proximal mappings for the $\infty$-norm on each group $G_i$.
These proximal mappings can be calculated with the help of projections onto the simplex (see~\cite{WLM13} for a more detailed description).
Figure~\ref{fig:M2-9-rec} shows some results for the Butterfly nebula after 100 steps of the BPSFP method with the constant stepsize and $\lambda = 10$ (note that in this case the dynamic stepsize is equal to the constant stepsize due to the special structure of the adjoint operator $P^T$).
In this example each color channel has been processed independently.
The projected image of the nebula consists of $122\times 512$ pixels and the reconstructed volume has $122\times 122\times 512$ voxels, i.e. the total number of variables for each color channel is about $7.6$ million.
The total run time for each color channel on an Intel\textsuperscript{\textregistered} Core\textsuperscript{\texttrademark} i7 CPU 960 3.20GHz was about 10 minutes.
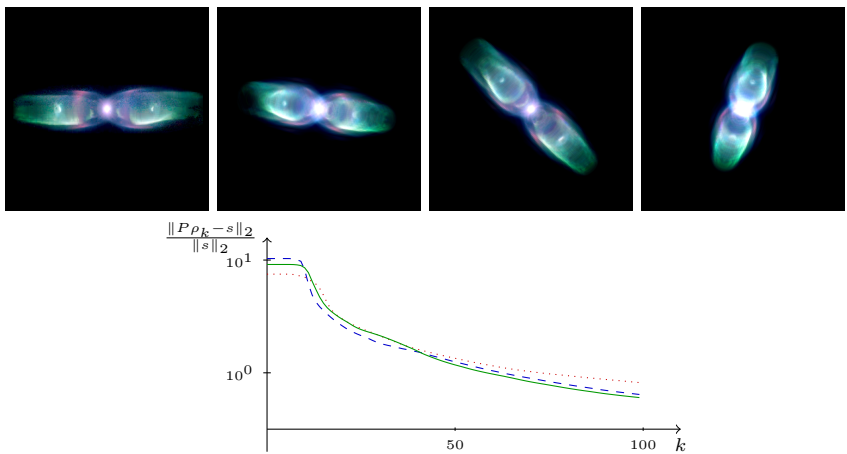
\begin{figure}[htb]
  \centering
  \includegraphics[width=2.7cm]{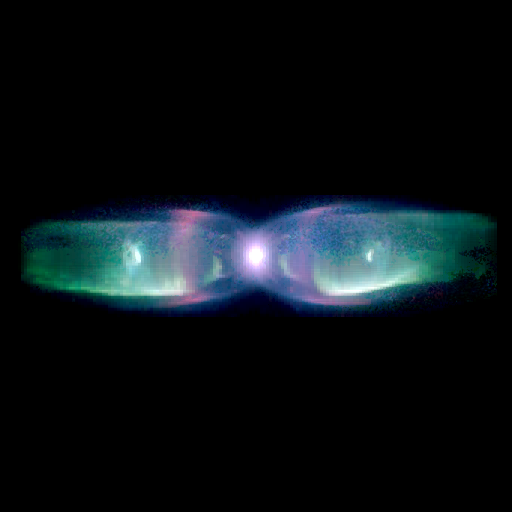}
  \includegraphics[width=2.7cm]{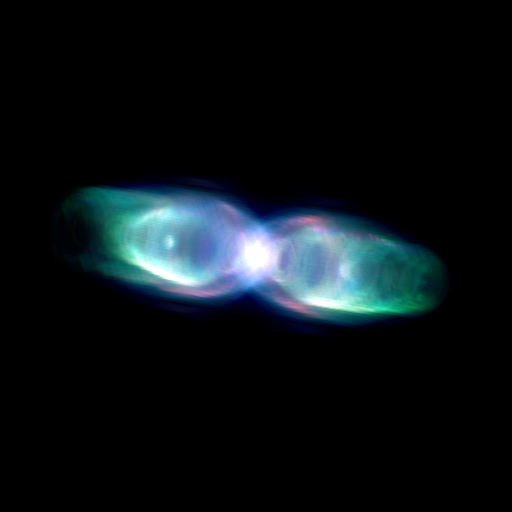} 
  \includegraphics[width=2.7cm]{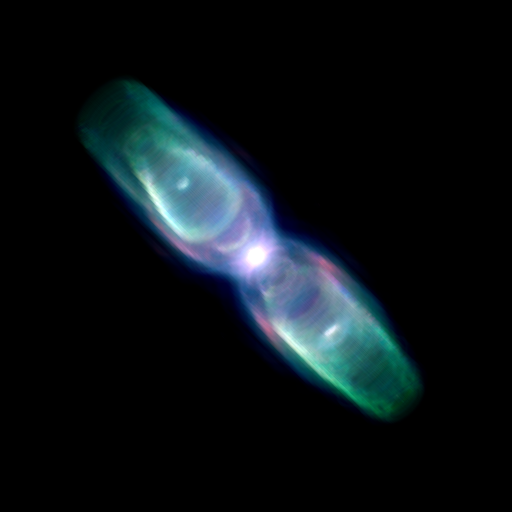}
  \includegraphics[width=2.7cm]{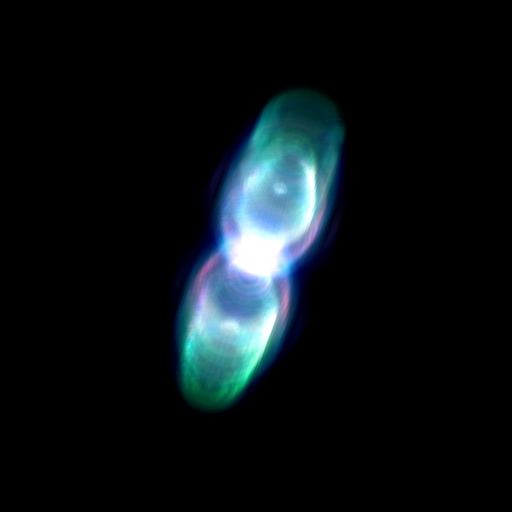}\\
  \begin{tikzpicture}[yscale=1.5,xscale=0.05]
    \draw[->] (0,-0.5) -- (110,-0.5)node[below]{\scriptsize $k$};
    \draw[->] (0,-0.7) -- (0,1.2)node[left]{\scriptsize $\tfrac{\norm[2]{P\rho_k-s}}{\norm[2]{s}}$};
    \foreach \i in {0,1}{
      \draw (1,\i) -- (-1,\i)node[left]{\tiny$10^{\i}$};
    }
    \foreach \i in {50,100}{
      \draw (\i,-0.485) -- (\i,-0.515)node[below]{\tiny$\i$};
    }
    
    \draw[red!80!black,dotted] plot file {data/M2-9/M2-9-red.dat};
    \draw[blue!80!black,dashed] plot file {data/M2-9/M2-9-blue.dat};
    \draw[green!60!black] plot file {data/M2-9/M2-9-green.dat};
  \end{tikzpicture}
  \caption{Top left: Reconstructed front view of the Saturn Nebula (compare with Figure~\ref{fig:M2-9}). Other pictures: New views reconstructed with the BPSFP method for a weighted $\ell^{1,\infty}$ term. Bottom: Logarithmic plot of the relative residual norms for all color channels (red: dotted, green: solid, blue: dashed).}
  \label{fig:M2-9-rec}
\end{figure}
Figure~\ref{fig:NGC7009-rec} shows some results for the Saturn Nebula (with the same algorithmic parameters).
In this example, the projected image of the nebula consists of $230\times 512$ pixels and the reconstructed volume has $230\times 230\times 512$ voxels, i.e. the total number of variables is about $27$ million.
The total run time for each color channel was about 20 minutes.
\begin{figure}[htb]
  \centering
  \includegraphics[width=2.7cm]{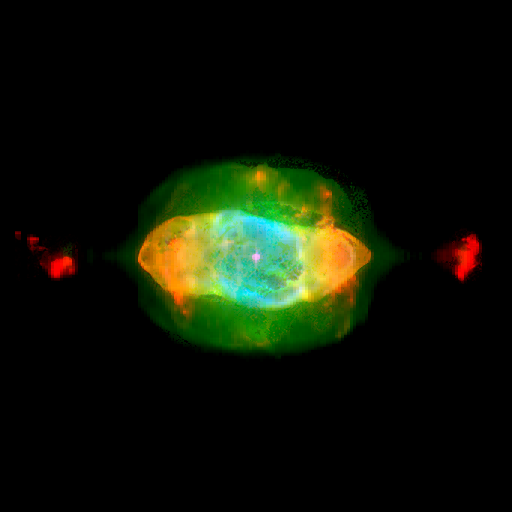}
  \includegraphics[width=2.7cm]{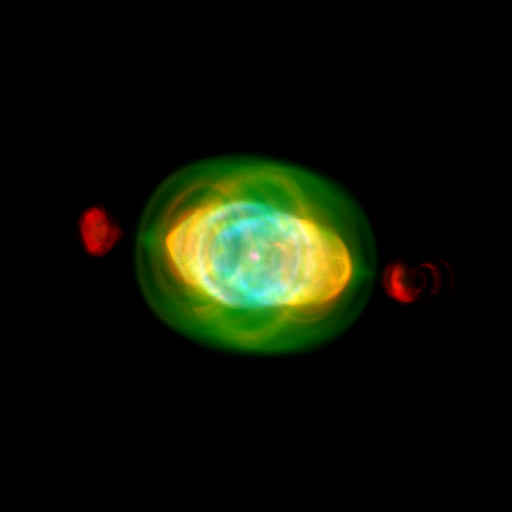} 
  \includegraphics[width=2.7cm]{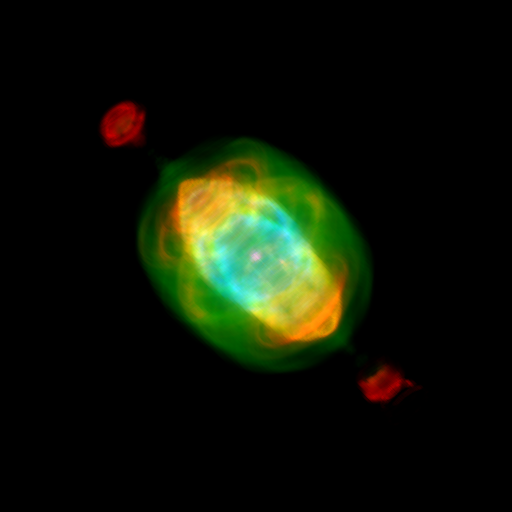}
  \includegraphics[width=2.7cm]{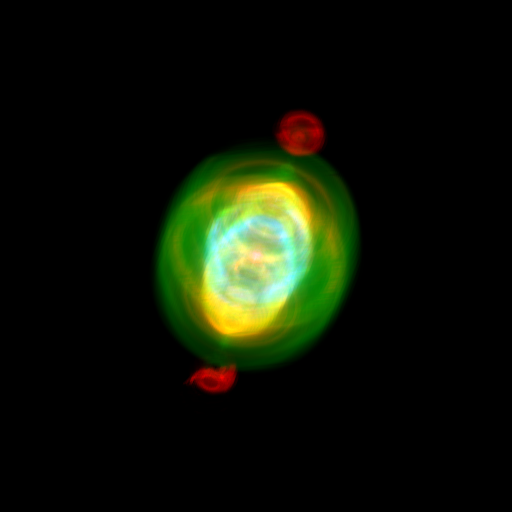}\\
  \begin{tikzpicture}[yscale=1.5,xscale=0.05]
    \draw[->] (0,0) -- (110,0)node[below]{\scriptsize $k$};
    \draw[->] (0,-0.2) -- (0,1.2)node[left]{\scriptsize $\tfrac{\norm[2]{P\rho_k-s}}{\norm[2]{s}}$};
    \foreach \i in {0,1}{
      \draw (1,\i) -- (-1,\i)node[left]{\tiny$10^{\i}$};
    }
    \foreach \i in {50,100}{
      \draw (\i,0.015) -- (\i,-0.015)node[below]{\tiny$\i$};
    }
    
    \draw[red!80!black,dotted] plot file {data/NGC7009/NGC7009-red.dat};
    \draw[blue!80!black,dashed] plot file {data/NGC7009/NGC7009-blue.dat};
    \draw[green!60!black] plot file {data/NGC7009/NGC7009-green.dat};
  \end{tikzpicture}
  \caption{Top left: Reconstructed front view of the Saturn Nebula (compare with Figure~\ref{fig:M2-9}). Other pictures: New views. Bottom: Logarithmic plot of the relative residual norms for all color channels (red: dotted, green: solid, blue: dashed).}
  \label{fig:NGC7009-rec}
\end{figure}
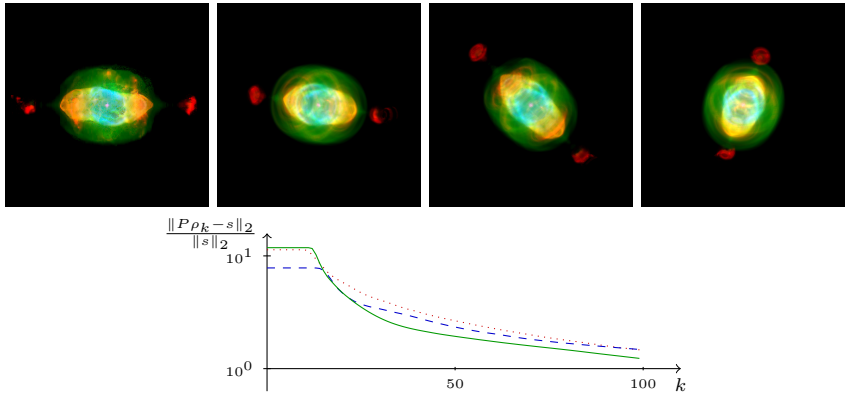

\subsection{$TV$ tomographic reconstruction with additional constraints}
\label{sec:tv-tomography}

The proposed BPSFP framework can handle multiple constraints and in this section we present such an example.
We consider a tomographic reconstruction problem. There one wants to reconstruct a two-dimensional function $u$ from its line-integrals (see~\cite{N86}).
In discretized form one has measurements $b^\delta\in\RR^m$ obtained by $Au^\dag = b$ with additional noise.
The matrix $A$ contains one row for every discretized line integral and one column for each pixel which is to be reconstructed.
As a matter of fact, this matrix is usually very sparse and also non-negative.
Moreover, the data $b^\delta$ is non-negative and also the unknown solution $u^\dag$ is non-negative.

To approximately solve the system $Ax=b^\delta$ in the underdetermined regime (i.e. less measurements than pixels) one often requires an additional regularization term and a popular choice is total variation regularization~\cite{SP08}.
In discrete form, the total variation is the 1-norm of the absolute value of the discrete gradient $\nabla u$ (i.e. the array of pointwise finite difference approximations of the partial derivatives).
If we assume that the noise level is known (approximately), i.e. the quantity $\delta = \norm[2]{b-b^\delta}$, then we consider the minimization problem
\ben
\min_u \norm[1]{|\nabla u|}\quad\text{s.t.}\quad \norm[2]{Au-b^\delta}\leq \delta.
\een
To apply the BPSFP framework we make two minor adaptions: First we introduce another variable $p$ and another constraint to obtain the equivalent problem
\ben
\begin{split}
  \min_{u,p} \norm[1]{|p|}\quad\text{s.t.}\quad
  \norm[2]{Au-b^\delta}&\leq \delta\\
   \nabla u &= p.
\end{split}
\een
Still, the objective is not strongly convex and we regularize the problem, as proposed at the end of section~\ref{sec:assumptions}, by adding the squared 2-norm of the variables.
With a further weight factor $\lambda>0$ we obtain
\begin{equation}
  \label{eq:TV-plain}
  \begin{split}
    \min_{u,p} \lambda\norm[1]{|p|} + \tfrac12\Big(\norm[2]{u}^2 +
    \norm[2]{|p|}^2\Big)\quad\text{s.t.}\quad
    \norm[2]{Au-b^\delta}&\leq \delta\\
    \nabla u &= p.
  \end{split}
\end{equation}
Now, the BPSFP framework can be applied straightforward: We consider the two constraints both as difficult constraints and treat them alternatingly.

The model can be enhanced further by incorporating more prior knowledge:
Since non-negativity is known, we can enforce this by adding a simple constraint
\begin{equation}
  \label{eq:TV-pos}
  \begin{split}
    \min_{u,p} \lambda\norm[1]{|p|} + \tfrac12\Big(\norm[2]{u}^2 +
    \norm[2]{|p|}^2\Big)\quad\text{s.t.}\quad
    \norm[2]{Au-b^\delta}&\leq \delta\\
    \nabla u &= p\\
    u & \geq 0.
  \end{split}
\end{equation}
This new constraint can be treated with almost no additional effort:
Just one Bregman projection onto the non-negative orthant is needed in every iteration.

In the special case of tomography one can assume that even more prior knowledge is available.
The data consists of parallel projections of the density from different angles. 
Since the projections amount to line integrals, and the parallel lines usually cover the whole region of interest, we obtain, by non-negativity of $u^\dag$, that the 1-norm of each parallel projection in $b^\delta$ is a good estimator of the one-norm of the solution.
Averaging over all these estimators provided by all angles further increases the accuracy.
Hence, we assume that an additional constraint $\norm[1]{u} = c$ is known and by non-negativity of $u$ we formulate this with the all-ones vector $\mathbbm{1}$ as
\begin{equation}
  \label{eq:TV-one}
  \begin{split}
    \min_{u,p} \lambda\norm[1]{|p|} + \tfrac12\Big(\norm[2]{u}^2 +
    \norm[2]{|p|}^2\Big)\quad\text{s.t.}\quad
    \norm[2]{Au-b^\delta}&\leq \delta\\
    \nabla u &= p\\
    u & \geq 0.\\
    \mathbbm{1}^Tu & = c
  \end{split}
\end{equation}
This additional simple constraint, indeed a hyperplane constraint, can be handled efficiently by another additional Bregman projection per iteration.

For a numerical example, we considered a phantom $u^\dag$ of $128\times 128$ pixels, 18 parallel projections (at angles $0^\circ,10^\circ,\dots,170^\circ$), each consisting of 136 parallel lines\footnote{We used the AIRtools package v1.0~\cite{HS12}, obtained from \url{http://www2.imm.dtu.dk/~pcha/AIRtools/} to build the data and the projection matrix.}.
This amounts to 2.448 measurements for 16.384 variables.
We produced data $b^\delta$ with 5\% Gaussian noise and set $\delta$ to the noise level.
We chose $\lambda=10$ and applied BPSFP with the dynamic stepsize for all three problems~\eqref{eq:TV-plain},~\eqref{eq:TV-pos}, and~\eqref{eq:TV-one}.
Figure~\ref{fig:tv-tomography} shows the result of BPSFP after 1.200 iterations.
Note that the incorporation of more prior knowledge does indeed increase the reconstruction quality.
Moreover, the additional constraints do not increase the computational time since the Bregman projections onto both constraints are simple and quick.

\begin{figure}[htb]
  \centering
  \begin{tabular}{cccc}
    \includegraphics[width=2.7cm]{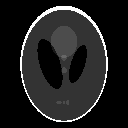} &
    \includegraphics[width=2.7cm]{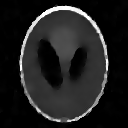} &
    \includegraphics[width=2.7cm]{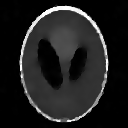} &
    \includegraphics[width=2.7cm]{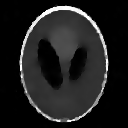}\\
    \small
    $u^\dag$&
    $\norm[2]{u^\dag - u^{\text{rec}}} = 10.4$ &
    $\norm[2]{u^\dag - u^{\text{rec}}} = 9.4$ & 
    $\norm[2]{u^\dag - u^{\text{rec}}} = 9.2$\\
  \end{tabular}
  \caption{From left to right: Original phantom.  Reconstruction with smal total variation~\eqref{eq:TV-plain}.  Additional positivity constraint~\eqref{eq:TV-pos}.  Positivity constraint and constraint for the 1-norm~\eqref{eq:TV-one}.}
  \label{fig:tv-tomography}
\end{figure}

Figure~\ref{fig:tv-errors} shows the violations of the constraints throughout the iterations for the BPSFP method for problems~\eqref{eq:TV-plain},~\eqref{eq:TV-pos}, and~\eqref{eq:TV-one}.
Note that the additional constraints also lead to slightly quicker reductions of the constraint violations and that the additional constraint one the 1-norm of the solution has a strong smoothing effect on the iteration.

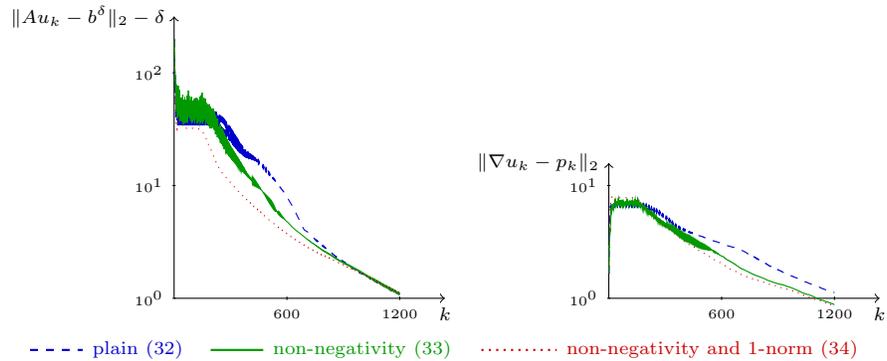
\begin{figure}[htb]
  \centering
  \begin{tikzpicture}[yscale=1.5,xscale=3]
    \draw[->] (0,0) -- (1.2,0)node[below]{\scriptsize $k$};
    \draw[->] (0,0) -- (0,2.5)node[left]{\scriptsize $\norm[2]{Au_k-b^\delta} - \delta$};
    \foreach \i in {0,1,2}{
      \draw (0.007,\i) -- (-0.007,\i)node[left]{\tiny$10^{\i}$};
    }
    \foreach \i/\ii in {0.5/600,1/1200}{
      \draw (\i,0.015) -- (\i,-0.015)node[below]{\tiny$\ii$};
    }
    \draw[blue!80!black,dashed] plot file {data/tomography/res1_plain.dat};
    \draw[green!60!black] plot file {data/tomography/res1_pos.dat};
    \draw[red!80!black,dotted] plot file {data/tomography/res1_one.dat};
  \end{tikzpicture}
  \begin{tikzpicture}[yscale=1.5,xscale=3]
    \draw[->] (0,0) -- (1.2,0)node[below]{\scriptsize $k$};
    \draw[->] (0,0) -- (0,1.2)node[left]{\scriptsize $\norm[2]{\nabla u_k-p_k}$};
    \foreach \i in {0,1}{
      \draw (0.007,\i) -- (-0.007,\i)node[left]{\tiny$10^{\i}$};
    }
    \foreach \i/\ii in {0.5/600,1/1200}{
      \draw (\i,0.015) -- (\i,-0.015)node[below]{\tiny$\ii$};
    }
    \draw[blue!80!black,dashed] plot file {data/tomography/res2_plain.dat};
    \draw[green!60!black] plot file {data/tomography/res2_pos.dat};
    \draw[red!80!black,dotted] plot file {data/tomography/res2_one.dat};
  \end{tikzpicture}\\
  \begin{tikzpicture}
    \draw[blue!80!black,dashed,thick] (0,0) --+ (0.7,0)node[right]{\scriptsize plain~\eqref{eq:TV-plain}}; 
    \draw[green!60!black,thick] (2.4,0) --+ (0.7,0)node[right]{\scriptsize  non-negativity~\eqref{eq:TV-pos}};
    \draw[red!80!black,dotted,thick] (6,0) --+ (1,0)node[right]{\scriptsize  non-negativity and 1-norm~\eqref{eq:TV-one}};
  \end{tikzpicture}
  \caption{Violations of the constraints for the three variants of
    $TV$ tomographic reconstruction.}
  \label{fig:tv-errors}
\end{figure}

\bibliography{./literature}
\bibliographystyle{plain}

\end{document}